\documentclass[12pt]{article}

\usepackage{geometry}

\usepackage{mathptmx}
\usepackage[T1]{fontenc}
\usepackage{framed,multirow}
\usepackage{amssymb,amsmath,amsthm,bm}
\usepackage{latexsym}
\usepackage{mathtools}
\usepackage{subcaption}
\usepackage{graphicx}
\usepackage{booktabs}
\usepackage{makecell}
\usepackage{enumitem} 
\usepackage{url}
\usepackage{xcolor}

\newtheorem{theorem}{Theorem}[]
\newtheorem{corollary}{Corollary}[]
\newtheorem{remark1}{Remark}[]
\theoremstyle{definition}
\newtheorem{definition}{Definition}[]
\theoremstyle{remark}
\newtheorem{example}{\bf Example}[]
\newenvironment{decomp}[1]
{\vspace{\topsep} \noindent{\bf #1.}}{\vspace{\topsep}}

\newcommand{\dx}{\Delta x}
\newcommand{\dy}{\Delta y}
\newcommand{\dz}{\Delta z}
\newcommand{\dt}{\Delta t}

\author{
Shumo Cui
\thanks{
Department of Mathematics and SUSTech International Center for Mathematics, Southern University of Science and Technology, Shenzhen 518055, China
({\tt cuism@sustech.edu.cn}).
}~{}
,~~
Shengrong Ding
\thanks{
Department of Mathematics, Southern University of Science and Technology, Shenzhen 518055, China
({\tt dingsr@sustech.edu.cn}).
}~{}
,~~
Kailiang Wu
\thanks{
Corresponding author. 
Department of Mathematics and SUSTech International Center for Mathematics, Southern University of Science and Technology, Shenzhen 518055, China; 
National Center for Applied Mathematics Shenzhen (NCAMS), Shenzhen 518055, China
({\tt wukl@sustech.edu.cn}).
The work of K. Wu. is supported in part by National Natural Science Foundation of China (grant No.~12171227).
}~{}
}

\title{{\Large \bf Is the Classic Convex Decomposition Optimal for Bound-Preserving Schemes in Multiple Dimensions?}}

\date{}

\begin{document}
\maketitle


\begin{abstract}
Since proposed in [X. Zhang and C.-W. Shu, {\em J. Comput. Phys.}, 229: 3091--3120, 2010], the Zhang--Shu framework has attracted extensive attention and motivated many bound-preserving (BP)  high-order discontinuous Galerkin and finite volume schemes for various hyperbolic equations. A key ingredient in the framework is the decomposition of the cell averages of the numerical solution into a convex combination of the solution values at certain quadrature points, which helps to rewrite high-order schemes as convex combinations of formally first-order schemes.  The classic convex decomposition originally proposed by Zhang and Shu has been widely used over the past decade. It was verified, only for the 1D quadratic and cubic polynomial spaces, that the classic decomposition is optimal in the sense of achieving the mildest BP CFL condition. {\em Yet, it remained unclear whether the classic decomposition is optimal in multiple dimensions.} In this paper, we find that the classic {\em multidimensional} decomposition based on the tensor product of Gauss--Lobatto and Gauss quadratures is generally {\em not} optimal, and we discover a novel alternative decomposition for the 2D and 3D polynomial spaces of total degree up to 2 and 3, respectively, on Cartesian meshes. Our new decomposition allows a larger BP time step-size than the classic one, and moreover, it is rigorously proved to be {\em optimal} to attain the mildest BP CFL condition, yet requires much fewer nodes. The discovery of such an optimal convex decomposition is highly nontrivial yet meaningful, as it may lead to an improvement of high-order BP schemes for a large class of hyperbolic or convection-dominated equations, at the cost of only a slight and local modification to the implementation code. Several numerical examples are provided to further validate the advantages of using our optimal decomposition over the classic one in terms of efficiency. 
\vspace{15mm}
\end{abstract}

\section{Introduction}

This paper is concerned with high-order robust numerical schemes for hyperbolic conservation laws 
\begin{equation}\label{HCL}
	\begin{cases}
	u_t + \nabla \cdot {\bm f} (u) =0,  \quad  & ({\bm x},t) \in \mathbb R^d \times \mathbb R^{+},
	\\
	u({\bm x},0) = u_0({\bm x}), \quad & {\bm x} \in \mathbb R^d,
\end{cases}
\end{equation}
where ${\bm x}$ denotes the spatial coordinate variable(s) in $d$-dimensional space, $t$ denotes the time, 
the conservative variable(s) $u$ takes values in $\mathbb R^m$, and the flux ${\bm f}=(f_1,\dots,f_d)$ takes values in 
$(\mathbb R^m)^d$. Our discussions in this paper can also be applicable to 
other related hyperbolic or convection-dominated equations. 

Solutions to the hyperbolic equations \eqref{HCL} often satisfy certain bounds, which constitute a convex invariant region $G \subset \mathbb R^m$.
When numerically solving such hyperbolic equations, it is highly desirable or even essential to 
preserve the intrinsic bounds, namely, to preserve the numerical solutions in the region $G$. 
In fact, if the numerical solutions go outside the bounds, for example, negative density or pressure is produced 
when solving the Euler equations, the discrete problem would become ill-posed due to the loss of hyperbolicity of the system, 
and may lead to the instability or breakdown of the numerical computation. 

As well known, first-order accurate monotone schemes, 
 such as the Godunov scheme, the Lax--Friedrichs scheme, and 
the Engquist--Osher scheme, are bound-preserving (BP) for scalar conservation laws and many hyperbolic systems. 
However, seeking high-order accurate BP schemes is rather nontrivial. 
In the pioneering work of \cite{zhang2010,zhang2010b}, Zhang and Shu proposed a general framework of designing high-order BP discontinuous Galerkin (DG) and  finite volume (FV) schemes for hyperbolic conservation laws on rectangular meshes,  later generalized to triangular meshes in \cite{zhang2012maximum}. 
Over the past decade, the Zhang--Shu framework has attracted extensive attention and been applied to various hyperbolic systems (e.g., \cite{xing2010positivity,zhang2011,WangZhangShuNing2012,zhang2012minimum,QinShuYang2016,Wu2017,jiang2018invariant,du2019high,wu2021minimum}) 
and convection-dominated equations  
 (e.g., \cite{zhang2012maximum-b,zhang2013maximum,ZHANG2017301,sun2018discontinuous,du2019maximum}). 
Recently, motivated by a series of BP works \cite{WuTangM3AS,Wu2017a,WuShu2018,WuShu2020NumMath} 
for magnetohydrodynamics, 
the geometric quasilinearization (GQL) framework was proposed in  \cite{Wu2021GQL} for studying BP problems involving nonlinear constraints. 
For more developments on high-order BP schemes, we refer the reader to the review papers \cite{zhang2011b,XuZhang2017,Shu2020AC} and some other BP techniques \cite{Xu2014,xiong2016parametrized,WuTang2015,guermond2017invariant}. 


An essential ingredient in the Zhang--Shu framework \cite{zhang2010,zhang2010b} is to decompose the cell averages of the numerical solution into a convex combination of the solution values at certain quadrature points.  
Based on such a convex decomposition, one can reformulate a one-dimensional (1D) or multidimensional high-order FV or DG scheme into a convex combination of formally 1D first-order schemes. This reformulation then leads to a sufficient condition for the BP property of the updated cell averages, combined with a 
simple local scaling limiter that can enforce the sufficient condition without losing the high-order accuracy \cite{zhang2010,zhang2010b}.

To illustrate the role of convex decomposition in Zhang--Shu's BP framework, let us consider  
a $(k+1)$th-order FV or DG scheme for 1D conservation laws with reconstructed or DG polynomials 
of degree $k$, denoted by $p_i(x)$, on cell $\Omega_{i}=[x_{i-\frac12}, x_{i+\frac12}]$. 
With forward 
Euler time discretization, the evolution equation of cell averages 
can be written as 
\begin{equation}\label{ZS-1D-cell-ave-evo}
	\bar {u}_i^{n+1} = \bar u_{i}^n - \lambda 
	\left(  \hat f( u_{i+\frac12}^{-}, u_{i+\frac12}^{+}  ) - \hat f( u_{i-\frac12}^{-}, u_{i-\frac12}^{+}  )  \right), 
\end{equation} 
where $\bar u_{i}^n$ is the cell average of $p_i(x)$ on $\Omega_{i}$ at time level $n$, 
$\lambda = \Delta t/\Delta x$ is the ratio of the temporal and spatial step-sizes, $u_{i-\frac12}^{+} = p_i( x_{i-\frac12} )$ and 
$u_{i+\frac12}^{-} = p_i( x_{i+\frac12} ) $ for all $i$. 
Here, $\hat f(\cdot,\cdot)$ is a BP numerical flux with which the first-order scheme is BP under  a suitable CFL condition 
$
	a  \lambda \le c_0,
$ 
where $a$ denotes the maximum characteristic speed, and $c_0$ is the maximum allowable CFL number for the corresponding first-order scheme. 
Note that the $L$-point Gauss--Lobatto quadrature with $L=\lceil \frac{k+3}{2}\rceil$ is exact for polynomials of degree up to $k$. This implies the following convex decomposition \cite{zhang2010}:
\begin{equation}\label{eq:GL1D}
	\bar u_i^n = \frac{1}{\Delta x} \int_{ x_{i-\frac12} }^{ x_{i+\frac12} } p_i(x) {\rm d} x 
	= {\omega}_1^{\tt GL}  u_{i-\frac12}^{+}  +  {\omega}_L^{\tt GL}  u_{i+\frac12}^{-} +
	\sum_{s = 2}^{L-1}  {\omega}_\ell^{\tt GL} p_i(  x_{i,s}^{\tt GL} ),
\end{equation}
where  $\{ {\omega}_s^{\tt GL} \}$ are the Gauss--Lobatto weights with ${\omega}_1^{\tt GL}={\omega}_L^{\tt GL}=1/(L(L-1))$, and $\{  x_{i,s}^{\tt GL} \}$ are the quadrature nodes on $\Omega_{i}$ with $x_{i,1}^{\tt GL} = x_{i-\frac12}$ and $x_{i,L}^{\tt GL} = x_{i+\frac12}$.    
Based on the decomposition \eqref{eq:GL1D}, Zhang and Shu \cite{zhang2010,zhang2010b} 
rewrote the scheme \eqref{ZS-1D-cell-ave-evo}
 as 
\begin{equation}\label{eq:1Dcd}
	\bar u_i^{n+1} =  {\omega}_1^{\tt GL}   \Pi_1 +   {\omega}_L^{\tt GL} \Pi_L  + \sum_{s = 2}^{L-1}  {\omega}_s^{\tt GL} p_i(  x_{i,s}^{\tt GL} ),
\end{equation}
where 
\begin{equation*}
	\Pi_1 :=  u_{i-\frac12}^+ - \frac{ \lambda  }{ {\omega}_1^{\tt GL}  } \left( 
	\hat f ( u_{i-\frac12}^+, u_{ i+\frac12  }^- ) 
	- \hat  f( u_{i-\frac12}^{-}, u_{i-\frac12}^{+}  )
	\right),
	 \qquad 
	\Pi_L :=  u_{i+\frac12}^- - \frac{ \lambda  }{ {\omega}_L^{\tt GL}  } \left( 
\hat f( u_{i+\frac12}^{-}, u_{i+\frac12}^{+}  )
- \hat f ( u_{i-\frac12}^+, u_{ i+\frac12  }^- ) 
\right)	
\end{equation*}
are of the same form as the three-point first-order scheme with a scaled time step-size. 
Thanks to the convex decomposition \eqref{eq:1Dcd} and the BP property of the first-order scheme, 
 if 
 we use a BP limiter \cite{zhang2010} to enforce  
 \begin{equation}\label{eq:1DPPcond}
 	p_{i}(  x_{i,s}^{\tt GL} ) \in G  \qquad \forall i, s, 
 \end{equation}
 then by the convexity of $G$, the high-order scheme \eqref{ZS-1D-cell-ave-evo} preserves $\bar u_j^{n+1} \in G$ under the CFL condition 
 \begin{equation}\label{1Dhst_CFL}
 	a  \lambda \le c_0 \omega_1^{\tt GL}. 
 \end{equation}

As we have seen, the convex decomposition \eqref{eq:GL1D} plays a critical role in 
constructing high-order BP schemes. It determines not only the theoretical BP CFL condition \eqref{1Dhst_CFL} of the resulting scheme, but also the points   \eqref{eq:1DPPcond} to perform the BP limiter. 
In fact, one may choose a different convex decomposition in the above analysis. 
In the 1D case, 
any type of quadrature rule, with weights all positive and nodes including the two endpoints, 
would give a feasible convex decomposition. 
However, different decomposition would affect the theoretical BP CFL condition and thus the computational costs. It is natural to ask what decomposition is optimal in the sense of 
achieving the mildest BP CFL condition. 
Zhang and Shu mentioned in \cite[Remark 2.7]{zhang2010} that they checked, for $k=2,3$, that their   
1D decomposition \eqref{eq:GL1D} is optimal. 
For $k\ge 4$, the optimality of the 1D decomposition \eqref{eq:GL1D} has not been proved yet.

This paper aims to make the first attempt at questing the optimal convex decomposition in the {\em multidimensional} cases. In the multiple dimensions, Zhang and Shu \cite{zhang2010,zhang2010b} proposed the classic 
convex decomposition based on the tensor product of the Gauss--Lobatto quadrature and the Gauss quadrature rules; see \eqref{eq:411}. 
As the 1D case, their decomposition has been an important foundation for constructing high-order BP multidimensional schemes. 
Over the past decade, the classic Zhang--Shu decomposition 
has been widely adopted in designing many high-order BP schemes for various hyperbolic or convection-dominated equations. It is natural to ask the following question:
\\[0.01mm]
\begin{center} 
	\parbox{0.888\textwidth}{\em {
			Is the classic convex decomposition optimal in multiple dimensions?}\\[0.01mm]}
\end{center}
In this work, we find, in the {\em multidimensional} cases, that the classic decomposition is 
generally not optimal for the $\mathbb P^k$ spaces, {\em i.e.}~the multivariate polynomial spaces of total degree up to $k$. {\em Seeking the optimal convex decomposition in the multidimensional cases is highly complicated and challenging.} 
In this paper, we restrict our attention to two commonly used spaces ($\mathbb P^2$ and $\mathbb P^3$), which are typically used in the third-order and fourth-order DG schemes, on the Cartesian meshes. For these polynomial spaces, {\em we discover a novel alternative decomposition, 
which is rigorously proved to be optimal, namely, to attain the mildest BP CFL condition, yet requires much fewer nodes.} 
Based on our novel optimal convex decomposition, we can establish more efficient high-order BP DG schemes in the Zhang--Shu framework, as it allows a notably larger BP time step-size than the classic one.  
The discovery of our optimal convex decomposition is highly nontrivial and may have a broad impact, as it would 
lead to an overall improvement of third-order and fourth-order BP schemes for  
a large class of hyperbolic or convection-dominated equations at the cost of only a slight and local modification to the implementation code. 
We will present several numerical examples to further validate the remarkable advantages of using our optimal decomposition over the classic one in terms of efficiency. 
It is worth mentioning that seeking the optimal convex decomposition for general $\mathbb P^k$ spaces ($k\ge 4$) in the multidimensional cases (on the Cartesian meshes or unstructured meshes) seems challenging and is still open. 
We hope the present paper could be helpful for motivating further discussions on this interesting problem 
in the future.

\section{General convex decomposition for 2D high-order BP schemes}\label{sec:2}


This section discusses the general feasible convex decomposition for constructing 2D high-order BP DG schemes within  
the Zhang--Shu framework. 

Let $\mathbb P^k$ denote the space of multivariate polynomials of total degree up to $k$. 
We consider the $(k+1)$th-order $\mathbb P^k$-based DG scheme with the forward Euler time discretization for solving 
 the 2D hyperbolic conservation laws
\begin{equation}\label{2DCL}
u_t+f_1(u)_x+f_2(u)_y = 0, \quad (x,y,t) \in \mathbb{R} \times \mathbb{R} \times \mathbb R^{+}.
\end{equation}
All our discussions in this paper are also valid for high-order strong-stability-preserving (SSP) time discretizations \cite{GottliebKetchesonShu2011}, as they are convex combinations of forward Euler step. 
Following the Zhang--Shu framework \cite{zhang2010,zhang2010b}, in order to design a BP DG scheme, we only need to ensure the cell averages within the region $G$. As long as the BP property of the updated cell averages is guaranteed, 
one may employ a simple BP limiter to enforce the pointwise bounds 
of the piecewise DG polynomial solutions without affecting the high-order accuracy \cite{zhang2010,zhang2010b}.

On a rectangular cell $\Omega_{ij}:=[x_{i-\frac12},x_{i+\frac12}]\times[y_{j-\frac12},y_{j+\frac12}]$, the evolution equation of cell averages for the $(k+1)$th-order DG scheme reads
\begin{equation}\label{eq:171}
	\bar {u}_{ij}^{n+1} = 
	\bar u_{ij}^n - 
	\frac{\Delta t}{\Delta x} \sum_{q=1}^Q \omega_q^{\tt G}
	\left[  \hat f_1 \big( u_{i+\frac12,q}^{-}, u_{i+\frac12,q}^{+}  \big) - \hat f_1 \big( u_{i-\frac12,q}^{-}, u_{i-\frac12,q}^{+}  \big)  \right] - 
	\frac{\Delta t}{\Delta y} \sum_{q=1}^Q \omega_q^{\tt G} 
	\left[  \hat f_2 \big( u_{q,j+\frac12}^{-}, u_{q,j+\frac12}^{+}  \big) - \hat f_2 \big( u_{q,j-\frac12}^{-}, u_{q,j-\frac12}^{+}  \big)  \right], 
\end{equation} 
where 
\begin{equation*}
	u_{i-\frac12,q}^{+} = p_{ij} \big( x_{i-\frac12}, y_{j,q}^{\tt G} \big), \quad 
u_{i+\frac12,q}^{-} = p_{ij} \big( x_{i+\frac12}, y_{j,q}^{\tt G} \big), \quad 
u_{q,j-\frac12}^{+} = p_{ij} \big( x_{i,q}^{\tt G}, y_{j-\frac12} \big), \quad 
u_{q,j+\frac12}^{-} = p_{ij} \big( x_{i,q}^{\tt G}, y_{j+\frac12} \big)
\end{equation*}
with $p_{ij}(x,y) \in \mathbb{P}^k$ denoting the DG solution polynomial on $\Omega_{ij}$ at time level $n$ satisfying
\begin{equation*}
	\bar {u}_{ij}^{n} = \frac{1}{\dx \dy} \int_{ x_{i-\frac12} }^{ x_{i+\frac12} } \int_{ y_{j-\frac12} }^{ y_{j+\frac12} } p_{ij}(x,y) \, {\rm d} y {\rm d} x, 
\end{equation*}
and $\{x_{i,q}^{\tt G}\}_{q=1}^Q$ and $\{y_{j,q}^{\tt G}\}_{q=1}^Q$ respectively denote 
the $Q$-point Gauss quadrature nodes in the intervals 
$[x_{i-\frac12},x_{i+\frac12}]$ and $[y_{j-\frac12},y_{j+\frac12}]$, with 
the corresponding quadrature weights $\{\omega_{q}^{\tt G}\}$ satisfying 
$\sum_{q=1}^Q \omega_{q}^{\tt G} =1$. For the $\mathbb P^k$-based DG scheme, $Q$ is typically taken as $k+1$, such that the quadrature has sufficiently high-order accuracy.

In \eqref{eq:171}, we take the numerical fluxes $\hat{f}_1$ and $\hat{f}_2$ as the BP numerical fluxes with which the corresponding 1D three-point first-order schemes are BP, {\rm i.e.}, for any $u_1,u_2,u_3 \in G$ it holds that 
\begin{equation}\label{1DBP}
 u_2 - \frac{ \dt }{ \dx} \left( \hat{f}_1(u_2,u_3) - \hat{f}_1(u_1,u_2) \right) \in G, \qquad 
 u_2 - \frac{\dt }{ \dy} \left( \hat{f}_2(u_2,u_3) - \hat{f}_2(u_1,u_2) \right) \in G \qquad 
\end{equation}
under a suitable CFL condition $\max\{a_1 \dt / \dx,a_2 \dt/\dy\} \le c_0$, where $a_1$ and $a_2$ denote the maximum characteristic speeds in $x$- and $y$-directions, and $c_0$ is the maximum allowable CFL number for the 1D first-order schemes. For example, typically $c_0=1$ for the Lax--Friedrichs flux \cite{zhang2010b,jiang2018invariant},  and $c_0=\frac12$ for the HLL and HLLC fluxes \cite{jiang2018invariant}.

%

\subsection{Feasible convex decomposition in 2D}

Similar to the 1D case (\ref{eq:GL1D}), the BP analysis and design of a 2D scheme  (\ref{eq:171}) also require certain 2D quadrature rule to decompose the cell average $\bar u_{ij}^n$ into a convex combination of the values of $p_{ij}$ at some points. 
A qualified 2D quadrature, which we call {\em feasible convex decomposition}, should satisfy three requirements, as defined below. 

\begin{definition}[Feasible convex decomposition in 2D]\label{def:2D_FCAD} 
	A 2D convex decomposition  
	\begin{equation}\label{2Ddecomp}
			\bar {u}_{ij}^{n} =  
		 \sum_{q=1}^Q \omega_{q}^{\tt G}  \left[  
		 \omega_1^- u_{i-\frac{1}{2}, q}^{+} +
		 \omega_1^+ u_{i+\frac{1}{2}, q}^{-} +
		 \omega_2^- u_{q, j-\frac{1}{2}}^{+} +
		 \omega_2^+ u_{q, j+\frac{1}{2}}^{-}
		 \right]
		 + \sum_{s=1}^S \omega_s p_{ij} ( x_{ij}^{(s)}, y_{ij}^{(s)} )
	\end{equation}
	is said to be {\em feasible} for the polynomial space $\mathbb P^k$, if it simultaneously satisfies 
	the following three conditions:
	\begin{enumerate}[label=(\roman*)]
		\item  the convex decomposition holds exactly for all $p \in \mathbb P^k$;
		\item  the weights $\{{\omega}_1^\pm, {\omega}_2^\pm, {\omega}_s\}$ are all positive (their summation equals one);
		\item  the internal node set $\mathbb S_{ij} := \{ ( x_{ij}^{(s)}, y_{ij}^{(s)}  ) \}_{s=1}^S \subset \Omega_{ij}$. 
	\end{enumerate}
\end{definition}

Based on the tensor product of the $L$-point Gauss quadrature (with $L=\lceil \frac{k+3}{2}\rceil$) and the $Q$-point Gauss--Lobatto quadrature, the cell average $\bar {u}_{ij}^{n}$ can be decomposed into a convex combination of point values of $p_{ij}$ as follows: 
\begin{align}\nonumber
\bar {u}_{ij}^{n} 
& = \frac{1}{\Delta x \Delta y} \int_{y_{j-\frac{1}{2}}}^{y_{j+\frac{1}{2}}} \int_{x_{i-\frac{1}{2}}}^{x_{i+\frac{1}{2}}} p_{i j}(x, y) \mathrm{d} x \mathrm{d} y
=
\sum_{s=1}^{L} \omega_{s}^{\tt {GL}}\left(\frac{1}{\Delta y} \int_{y_{j-\frac{1}{2}}}^{y_{j+\frac{1}{2}}} p_{i j}\left(x_{i, s}^{\tt {GL}}, y\right) \mathrm{d} y\right) =
\sum_{s=1}^{L} \sum_{q=1}^{Q} \omega_{s}^{\tt {GL}} \omega_{q}^{\tt {G}}p_{i j}\left(x_{i, s}^{\tt {GL}}, y_{j, q}^{\tt {G}}\right) 
 \\ \label{eq:GGL_x}
&  =
\sum_{q=1}^Q \omega_{q}^{\tt G} \omega_1^{\tt GL} \left[
		 u_{i-\frac{1}{2}, q}^{+} + u_{i+\frac{1}{2}, q}^{-}
\right] 
 +
\sum_{s=2}^{L-1} \sum_{q=1}^{Q}  \omega_s^{\tt GL}  \omega_q^{\tt G} 
p_{i j}\left(x_{i,s}^{\tt GL}, y_{j,q}^{\tt G} \right)
.
\end{align}
Similarly, 
by applying the quadrature rules in a different order, one obtains
\begin{equation}\label{eq:GGL_y}
\bar {u}_{ij}^{n} = 
\sum_{q=1}^Q \omega_{q}^{\tt G} \omega_1^{\tt GL} \left[
		 u_{q, j-\frac{1}{2}}^{+} +
		 u_{q, j+\frac{1}{2}}^{-}
\right] 
 +
\sum_{s=2}^{L-1} \sum_{q=1}^{Q}  \omega_s^{\tt GL}  \omega_q^{\tt G} 
p_{i j}\left(x_{i,q}^{\tt G},y_{j,s}^{\tt GL}  \right)
.
\end{equation}

\begin{decomp}{Zhang--Shu classic convex decomposition}
In \cite{zhang2010,zhang2010b}, Zhang and Shu proposed the classic convex decomposition by using the tensor-product decomposition formulas (\ref{eq:GGL_x}) and (\ref{eq:GGL_y}):
\begin{align}\nonumber
\bar{u}_{i j}^{n} & 
= 
\kappa_1 \cdot \bar {u}_{ij}^{n} + \kappa_2 \cdot \bar {u}_{ij}^{n}
=
\kappa_1 \cdot \textrm{equation (\ref{eq:GGL_x})} + 
\kappa_2 \cdot \textrm{equation (\ref{eq:GGL_y})} \\ \label{eq:411}
& =
\sum_{q=1}^Q \omega_{q}^{\tt G} \omega_1^{\tt GL} \left[
		 \kappa_1 u_{i-\frac{1}{2}, q}^{+} +
		 \kappa_1 u_{i+\frac{1}{2}, q}^{-} +
		 \kappa_2 u_{q, j-\frac{1}{2}}^{+} +
		 \kappa_2 u_{q, j+\frac{1}{2}}^{-}
\right] 
 +
\sum_{s=2}^{L-1} \sum_{q=1}^{Q}  \omega_s^{\tt GL}  \omega_q^{\tt G} 
\left[ 
\kappa_1  p_{i j}\left(x_{i,s}^{\tt GL}, y_{j,q}^{\tt G} \right) + 
\kappa_2  p_{i j}\left(x_{i,q}^{\tt G},y_{j,s}^{\tt GL}  \right)
\right]
\end{align}
with 
\begin{equation*}
	 \kappa_1 := \frac{\frac{a_{1}}{\dx}}{\frac{a_1}{\dx}+\frac{a_2}{\dy}}, \qquad
 \kappa_2 := \frac{\frac{a_{2}}{\dy}}{\frac{a_1}{\dx}+\frac{a_2}{\dy}}.
\end{equation*}
This classic convex decomposition has been widely used over the past decase.
\end{decomp}

\begin{decomp}{Jiang--Liu convex decomposition}
In \cite{jiang2018invariant}, Jiang and Liu used a simpler convex decomposition:
\begin{align}\nonumber
\bar{u}_{i j}^{n} 
& =
\frac12 \cdot \bar {u}_{ij}^{n} + \frac12 \cdot \bar {u}_{ij}^{n}
=
\frac12 \cdot \textrm{equation (\ref{eq:GGL_x})} + 
\frac12 \cdot \textrm{equation (\ref{eq:GGL_y})} \\ \label{eq:433}
& = 
\sum_{q=1}^Q \frac{\omega_{q}^{\tt G}\omega_1^{\tt GL}}{2}\left[
		 u_{i-\frac{1}{2}, q}^{+} +
		 u_{i+\frac{1}{2}, q}^{-} +
		 u_{q, j-\frac{1}{2}}^{+} +
		 u_{q, j+\frac{1}{2}}^{-}
\right] 
+
\sum_{s=2}^{L-1} \sum_{q=1}^{Q}  \frac{\omega_s^{\tt GL}  \omega_q^{\tt G} }{2}
\left[ 
 p_{i j}\left(x_{i,s}^{\tt GL}, y_{j,q}^{\tt G} \right) + 
 p_{i j}\left(x_{i,q}^{\tt G},y_{j,s}^{\tt GL}  \right)
\right].
\end{align}
\end{decomp}

\begin{remark1}
Both the Zhang--Shu decomposition \eqref{eq:411} and the Jiang--Liu decomposition \eqref{eq:433} are examples of 2D feasible convex decomposition, and they share the same (classic) internal node set
\begin{equation}\label{eq:450}
	\mathbb{S}^{\tt Zhang-Shu}_{ij} = \bigcup_{s=2}^{L-1}\bigcup_{q=1}^{Q}\left\{(x_{i,s}^{\tt GL}, y_{j,q}^{\tt G} ),(x_{i,q}^{\tt G}, y_{j,s}^{\tt GL} )\right\}.
\end{equation}
\end{remark1}

\subsection{BP conditions via general convex decomposition}

Motivated by \cite{zhang2010,zhang2010b,jiang2018invariant}, this subsection 
studies the BP conditions for the scheme (\ref{eq:171}) via the general feasible convex decomposition \eqref{2Ddecomp}. 
One can rewrite \eqref{2Ddecomp} as 
\begin{equation}\label{eq:454}
	\bar u_{ij}^{n} = 
\sum_{q=1}^Q   \omega_q^{\tt G} \,
\widehat \omega_1 \left(  u_{i-\frac12,q}^+ +  u_{i+\frac12,q}^- \right) 
+ 
\sum_{q=1}^Q   \omega_q^{\tt G} \,
\widehat \omega_2 \left(  u_{q,j-\frac12}^+ +  u_{q,j+\frac12}^- \right) + \Pi,
\end{equation}
where $\widehat \omega_1 := \min\{ \omega_1^-, \omega_1^+ \}$, $\widehat \omega_2 := \min\{ \omega_2^-, \omega_2^+ \}$, and 
\begin{equation}\label{eq:Pi}
	\Pi := 
\sum_{q=1}^Q  \omega_q^{\tt G} 
\left[ 
(\omega_1^- -\widehat \omega_1) u_{i-\frac12,q}^+ + 
(\omega_1^+ -\widehat \omega_1) u_{i+\frac12,q}^-
+ 
(\omega_2^--\widehat \omega_2) u_{q,j-\frac12}^+ + 
(\omega_2^+-\widehat \omega_2) u_{q,j+\frac12}^- \right]
+\sum_{s=1}^S \omega_s p_{ij}( x_{ij}^{(s)}, y_{ij}^{(s)}  ).
\end{equation}
By using a local scaling limiter \cite{zhang2010,zhang2010b}, one can enforce the DG solution polynomial $p_{ij}$ to satisfy the desired bounds on the boundary nodes:
\begin{equation}\label{eq:478a}
u_{i-\frac12,q}^{+}  \in G, \quad
u_{i+\frac12,q}^{-}  \in G, \quad
u_{q,j-\frac12}^{+}  \in G, \quad
u_{q,j+\frac12}^{-}  \in G, \quad
q = 1,\dots,Q, \quad \forall i,j,
\end{equation}
and 
on the internal nodes:
\begin{equation}\label{eq:478b}
p_{ij}( x_{ij}^{(s)}, y_{ij}^{(s)}) \in G, \quad s = 1,\dots,S, \quad \forall i,j.
\end{equation}
Noting that (\ref{eq:Pi}) expresses $\Pi$ as a convex combination of the values in (\ref{eq:478a}) and (\ref{eq:478b}), we conclude that $\Pi \in G$, because $G$ is convex. 
Substituting the decomposition \eqref{eq:454} into \eqref{eq:171}, one can rewrite the scheme \eqref{eq:171} as 
\begin{equation}\label{eq:476}
		\bar {u}_{ij}^{n+1}  = \sum_{q=1}^Q    \omega_q^{\tt G}  \,
		\widehat \omega_1 (  H_{i+\frac12,q}^- + H_{i-\frac12,q}^+  ) + 
		 \sum_{q=1}^Q    \omega_q^{\tt G}  \,
		\widehat \omega_2 ( H_{q,j+\frac12}^- + H_{q,j-\frac12}^+ ) + \Pi
\end{equation} 
with
\begin{align*}
	H_{i+\frac12,q}^- &= u_{i+\frac12,q}^- - \frac{\Delta t}{\widehat \omega_1 \Delta x} 
	\left( \hat f_1 \big( u_{i+\frac12,q}^-, u_{i+\frac12,q}^+ \big)  
	- \hat f_1 \big( u_{i-\frac12,q}^+, u_{i+\frac12,q}^- \big)  \right),
	\\
	H_{i-\frac12,q}^+ & = u_{i-\frac12,q}^+ - \frac{\Delta t}{\widehat \omega_1 \Delta x} 
	\left(   \hat f_1 \big( u_{i-\frac12,q}^+, u_{i+\frac12,q}^- \big) 
	- \hat f_1 \big( u_{i-\frac12,q}^-, u_{i-\frac12,q}^+ \big)   \right),
	\\
	H_{q,j+\frac12}^- &= u_{q,j+\frac12}^- - \frac{\Delta t}{\widehat \omega_2 \Delta y} 
	\left( \hat f_2 \big( u_{q,j+\frac12}^-, u_{q,j+\frac12}^+ \big)  
	- \hat f_2 \big( u_{q,j-\frac12}^+, u_{q,j+\frac12}^- \big)  \right),
	\\
	H_{q,j-\frac12}^+ &= u_{q,j-\frac12}^+ - \frac{\Delta t}{\widehat \omega_2 \Delta y} 
	\left(  \hat f_2 \big( u_{q,j-\frac12}^+, u_{q,j+\frac12}^- \big) 
	- \hat f_2 \big( u_{q,j-\frac12}^-, u_{q,j-\frac12}^+ \big) \right),
\end{align*}
which are formally 1D three-point first-order schemes \eqref{1DBP} that satisfy  
\begin{equation*}
	H_{i+\frac12,q}^-\in G, \quad  H_{i-\frac12,q}^+ \in G, \quad H_{q,j+\frac12}^- \in G, \quad H_{q,j-\frac12}^+ \in G,
\end{equation*}
under the CFL type conditions 
\begin{equation}\label{eq:tempCFL}
a_1 {\Delta t}\le {\widehat \omega_1  \Delta x}, \quad a_2 {\Delta t}\le { \widehat \omega_2  \Delta y}. 
\end{equation}
Because (\ref{eq:476}) is a convex combination form, by the convexity of $G$ we conclude that $\bar {u}_{ij}^{n+1} \in G$ under the conditions \eqref{eq:tempCFL}, which are equivalent to the following BP CFL condition \eqref{2Dhst_CFL_all}. 
In summary, we arrive at the following theorem.


\begin{theorem}[BP via general convex decomposition]\label{thm:CFL}
If there is a 2D feasible convex decomposition in the form of \eqref{eq:QR2D} and the solution polynomial $p_{ij}$ satisfies (\ref{eq:478a}) and (\ref{eq:478b}) for all $i$ and $j$, 
then the high-order scheme \eqref{eq:171} preserves 
$\bar u_{ij}^{n+1}\in G$ under the BP CFL condition
\begin{equation}\label{2Dhst_CFL_all}
\Delta t \le c_0 \min \left\{ \frac{\omega_1^- \Delta x}{a_1}, \frac{\omega_1^+ \Delta x}{a_1} , \frac{\omega_2^- \Delta y}{a_2}, \frac{\omega_2^+ \Delta y}{a_2} \right\}.
\end{equation}
\end{theorem}

As direct consequences of 
Theorem \ref{thm:CFL}, we have the following two corollaries.

\begin{corollary}[BP via Zhang--Shu convex decomposition]
If for all $i$ and $j$, the solution polynomial $p_{ij}$ satisfies (\ref{eq:478a}) and
$
p_{ij}(x,y) \in G$ for all $(x,y) \in \mathbb{S}^{\tt Zhang-Shu}_{ij}
$, then the high-order scheme \eqref{eq:171} preserves 
$\bar u_{ij}^{n+1}\in G$ under the BP CFL condition
\begin{equation}\label{eq:CFL-ZS}
	\left( \frac{a_1}{\dx}+\frac{a_2}{\dy} \right)\Delta t \le \omega_1^{\tt GL}\,c_0 
=\frac{c_0}{ L(L-1) } \quad \mbox{with} \quad L=\left \lceil \frac{k+3}{2} \right\rceil.
\end{equation}
\end{corollary}

\begin{corollary}[BP via Jiang--Liu convex decomposition]
If for all $i$ and $j$, the solution polynomial $p_{ij}$ satisfies (\ref{eq:478a}) and 
$
p_{ij}(x,y) \in G$ for all $(x,y) \in \mathbb{S}^{\tt Zhang-Shu}_{ij}
$, 
then the high-order scheme \eqref{eq:171} preserves 
$\bar u_{ij}^{n+1}\in G$ under the BP CFL condition
\begin{equation*}
	2 \max \left\{\frac{a_1}{\dx},\frac{a_2}{\dy}\right\}\Delta t \le \omega_1^{\tt GL}\,c_0 =\frac{c_0}{ L(L-1) } \quad \mbox{with} \quad L=\left \lceil \frac{k+3}{2} \right\rceil.
\end{equation*}
\end{corollary}




\section{Optimal 2D convex decomposition for $\mathbb{P}^\textnormal{2}$ and $\mathbb{P}^\textnormal{3}$ on rectangular cells}\label{sec:3}

As we have seen, a convex decomposition like (\ref{2Ddecomp}) plays a critical role in constructing 2D high-order BP schemes: the choice of decomposition, in particular, the corresponding weights $\{\omega_1^-, \omega_1^+, \omega_2^-,  \omega_2^+\}$ affect the resulting BP CFL condition (\ref{2Dhst_CFL_all}). While the feasible convex decomposition approaches are not unique, it is natural to seek the {\em optimal convex decomposition} such that the resulting BP CFL condition \eqref{2Dhst_CFL_all} is mildest, i.e., it maximizes 
\begin{equation*}
	\min  \left\{ \frac{\omega_1^- \Delta x}{a_1}, \frac{\omega_1^+ \Delta x}{a_1} , \frac{\omega_2^- \Delta y}{a_2}, \frac{\omega_2^+ \Delta y}{a_2} \right\}.
\end{equation*}

Seeking such an optimal convex decomposition in 2D is challenging. 
We find that the classic Zhang--Shu decomposition 
\eqref{eq:411} and the Jiang--Liu decomposition \eqref{eq:433} both are generally 
not optimal for the $\mathbb P^k$ spaces. 
Moreover, we discover the following novel convex decomposition on $\Omega_{ij}:=[x_{i-\frac12},x_{i+\frac12}]\times[y_{j-\frac12},y_{j+\frac12}]$ for $\mathbb{P}^2$ and $\mathbb{P}^3$.

\begin{decomp}{Optimal 2D convex decomposition} 
	For $p_{ij}\in\mathbb{P}^2$ or $\mathbb P^3$, 
	the cell average $\bar u_{ij}^n$ has the following convex decomposition 
\begin{equation}\label{eq:QR2D}
\bar u_{ij}^n =
\frac{\mu_1}{2} \sum_{q=1}^Q \omega_{q}^{\tt G} \left[
  u_{i-\frac12,q}^{+} +
  u_{i+\frac12,q}^{-} 
\right] +
\frac{\mu_2}{2} \sum_{q=1}^Q \omega_{q}^{\tt G} \left[
  u_{q,j-\frac12}^{+} +
  u_{q,j+\frac12}^{-}
\right] +
\omega\sum_{s} p_{ij}\left( \hat{x}_s,\hat{y}_s \right),
\end{equation}
with the internal nodes
\begin{equation}\label{eq:226}
\mathbb{S}_{ij}^{\tt optimal} = \{(\hat{x}_s,\hat{y}_s)\} =
\begin{dcases}
\; \left( 
x_i, 
y_j \pm \frac{\dy}{2\sqrt{3}}\sqrt{\frac{\phi_*-\phi_2}{\phi_*}}
\right),
& \text{if} \; \phi_1 \ge \phi_2, \\
\; \left(
x_i \pm \frac{\dx}{2\sqrt{3}}\sqrt{\frac{\phi_*-\phi_1}{\phi_*}} ,
y_j
\right),
& \text{if} \; \phi_1 < \phi_2,
\end{dcases}
\end{equation}
where
\begin{equation*}
	\phi_1 = \frac{ a_1}{\dx}, \quad
\phi_2 = \frac{a_2}{\dy}, \quad
\phi_* = \max\{\phi_1,\phi_2\}, \quad
\psi = \phi_1+\phi_2+2\phi_*, \quad
\mu_1 = \frac{\phi_1}{\psi}, \quad
\mu_2 = \frac{\phi_2}{\psi}, \quad
\omega = \frac{\phi_*}{\psi}.
\end{equation*}
\end{decomp}

It can be verified that the proposed 2D convex decomposition (\ref{eq:QR2D}) is feasible and optimal (see Theorem \ref{thm:OPT2DP2}) for both $\mathbb{P}^2$ and $\mathbb{P}^3$. 
As a direct consequence of 
Theorem \ref{thm:CFL}, we have following corollary. 

\begin{corollary}[BP via optimal convex decomposition]
	If for all $i$ and $j$, the solution polynomial $p_{ij}$ satisfies (\ref{eq:478a}) and 
$
p_{ij}(x,y) \in G$ for all $(x,y) \in \mathbb{S}^{\tt optimal}_{ij}
$, then the $\mathbb{P}^2$-based or $\mathbb{P}^3$-based high-order DG scheme \eqref{eq:171} preserves 
	$\bar u_{ij}^{n+1}\in G$ under the BP CFL condition
\begin{equation}\label{eq:CFL2D}
\left[  2\frac{a_1}{\dx}+2\frac{a_2}{\dy}+ 4 \max\left\{ 
\frac{a_1}{\dx},\frac{a_2}{\dy}\right\} \right] \Delta t \le c_0.
\end{equation}
\end{corollary}


In Table \ref{tab:330}, we list and compare the corresponding BP CFL conditions and the internal node sets of decompositions (\ref{eq:QR2D}), (\ref{eq:411}) and (\ref{eq:433}) for $\mathbb{P}^2$ and $\mathbb{P}^3$. We observe that 
our novel convex decomposition (\ref{eq:QR2D}) has some remarkable advantages, as summarized in the following remarks. 

\begin{remark1}[Advantage in mildest BP CFL condition]\label{rem1}
One can observe from Table \ref{tab:330} that the proposed convex decomposition (\ref{eq:QR2D}) achieves a notably milder BP CFL condition than the existing ones, {\rm i.e.}, our BP CFL condition (\ref{eq:CFL2D}) is weaker than (\ref{eq:411}) and (\ref{eq:433}) respectively obtained via the Zhang--Shu and Jiang--Liu convex decompositions.  In fact, for the $\mathbb{P}^2$ or $\mathbb{P}^3$ space, no other 2D feasible convex decomposition can achieve an even milder BP CFL condition than (\ref{eq:CFL2D}), {\rm i.e.}, the proposed convex decomposition (\ref{eq:QR2D}) is {\em optimal}. 
This will be theoretically proved in Theorem \ref{thm:OPT2DP2}. 
\end{remark1}

\begin{remark1}[Advantage in fewer nodes]\label{rem2}
The internal node set $\mathbb{S}_{ij}^{\tt optimal}$ of the optimal convex decomposition (\ref{eq:QR2D}) contains only two nodes, which merge to a single node  $(x_i,y_j)$ in case of $\phi_1=\phi_2$. In comparison, the classic convex decomposition (\ref{eq:411}) or (\ref{eq:433}) needs much more internal nodes (approximately $2Q(L-2)$ in total). These two internal node sets, $\mathbb{S}_{ij}^{\tt Zhang-Shu}$ and $\mathbb{S}_{ij}^{\tt optimal}$, are shown in Figure \ref{fig:374} for further comparison. When the local scaling BP limiter is performed at the internal nodes in all computational cells, using the optimal convex decomposition (\ref{eq:QR2D}) reduces the computational cost of the BP limiting procedure. 
\end{remark1}

\begin{remark1}[Easy implementation]
It is worth emphasizing that one 
 only requires a slight and local modification to an existing code to enjoy the above-mentioned advantages of our optimal convex decomposition. Specifically, one only needs to replace the classic internal node set $\mathbb{S}_{ij}^{\tt Zhang-Shu}$ with the optimal internal node set $\mathbb{S}_{ij}^{\tt optimal}$ 
 in the BP limiting procedure, and then the theoretical BP CFL condition is improved to \eqref{eq:CFL2D}. 
\end{remark1}

\begin{table}[h]
\centering
\caption{The theoretical BP CFL conditions and the internal node sets of the optimal convex decomposition (\ref{eq:QR2D}) and the convex decompositions (\ref{eq:411}) and (\ref{eq:433}) in 2D for the $\mathbb{P}^2$ and $\mathbb{P}^3$ spaces.}\label{tab:330}
\renewcommand\arraystretch{1.7}
\begin{tabular}{ lrrrr} 
 \toprule[1.5pt]
  & BP CFL condition & BP CFL condition & Internal & Internal \\
 & general case & $\frac{\dx}{a_1} = \frac{\dy}{a_2} = h$ & $\mathbb{P}^2$ nodes & $\mathbb{P}^3$ nodes \\ 
 
 \midrule[1.5pt]
 
Optimal & 
$\left[ 2 \frac{a_1}{\dx}+2 \frac{a_2}{\dy}+4 \max\left\{ 
\frac{a_1}{\dx},\frac{a_2}{\dy} \right\} \right] \Delta t \le c_0$ 
& $\Delta t \le \frac{c_0}{8} h$ & 1 $\sim$ 2 & 1 $\sim$ 2 \\ 
 Zhang \& Shu \cite{zhang2010} & $\left[6\frac{a_1}{\dx}+6\frac{a_2}{\dy}\right]\Delta t \le c_0$ & $\Delta t \le \frac{c_0}{12} h$ & 5 & 8 \\ 
 Jiang \& Liu \cite{jiang2018invariant} & $\left[12\max\left\{\frac{a_1}{\dx},\frac{a_2}{\dy} \right\}\right]\Delta t \le c_0$ & $\Delta t \le \frac{c_0}{12}h$ & 5 & 8\\ 

\bottomrule[1.5pt]
\end{tabular}
\end{table}

\begin{figure}
\centerline{
\begin{subfigure}{0.24\textwidth}
    \includegraphics[width=\textwidth]{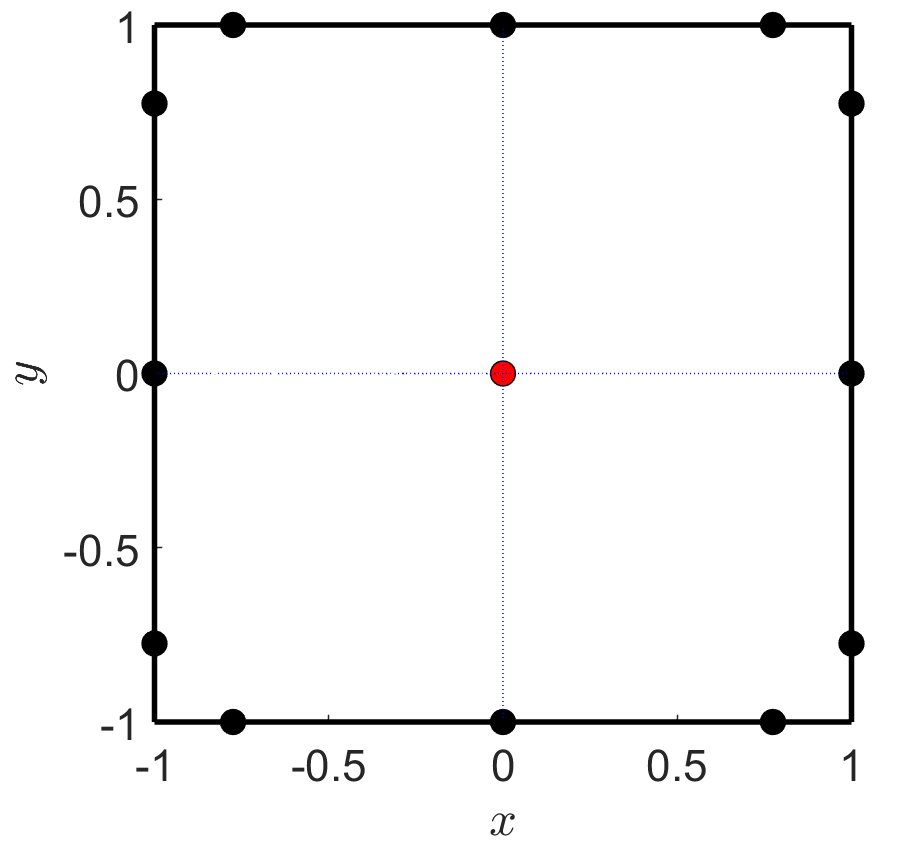}
    \caption{Boundary nodes (black) and optimal internal nodes $\mathbb{S}_{ij}^{\tt optimal}$ (red) for $\mathbb{P}^2$.}
\end{subfigure}
\hfill
\begin{subfigure}{0.24\textwidth}
    \includegraphics[width=\textwidth]{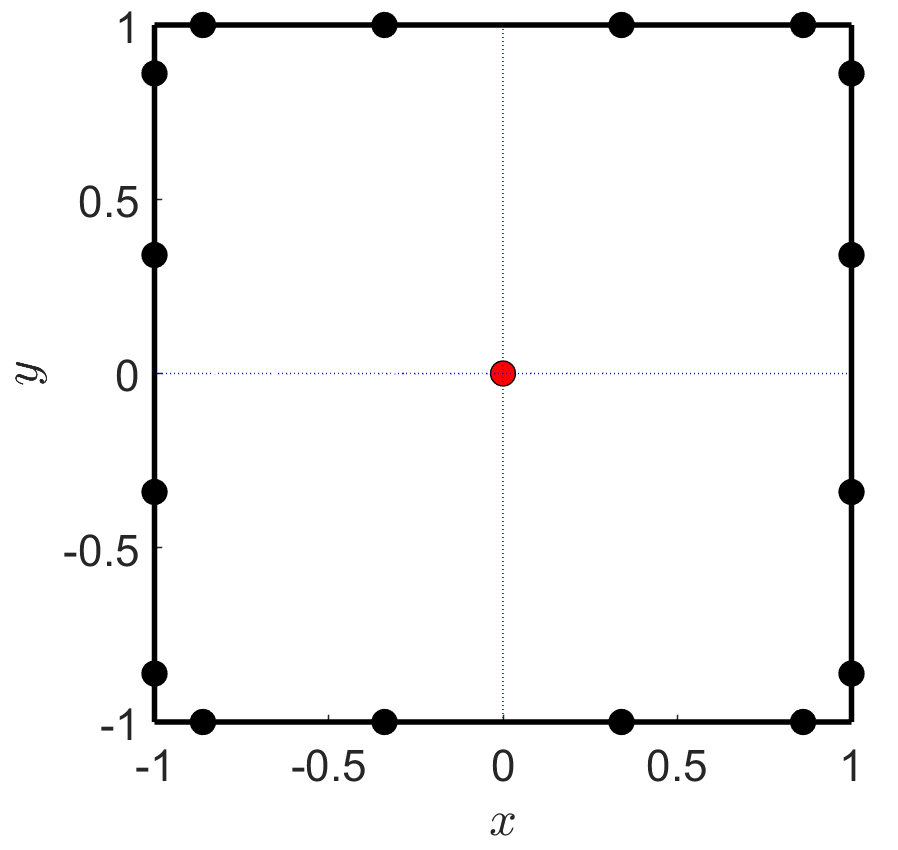}
    \caption{Boundary nodes (black) and optimal internal nodes $\mathbb{S}_{ij}^{\tt optimal}$ (red) for $\mathbb{P}^3$.}
\end{subfigure}
\hfill
\begin{subfigure}{0.24\textwidth}
    \includegraphics[width=\textwidth]{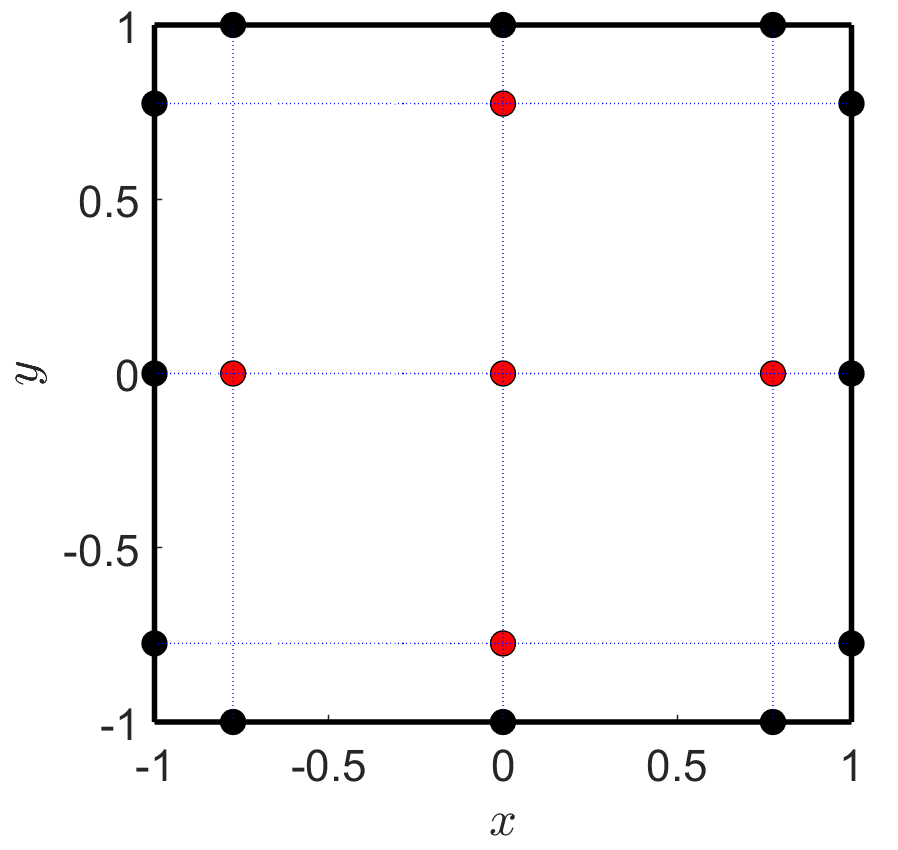}
    \caption{Boundary nodes (black) and classic internal nodes $\mathbb{S}_{ij}^{\tt Zhang-Shu}$ (red) for $\mathbb{P}^2$.}
\end{subfigure}
\hfill
\begin{subfigure}{0.24\textwidth}
    \includegraphics[width=\textwidth]{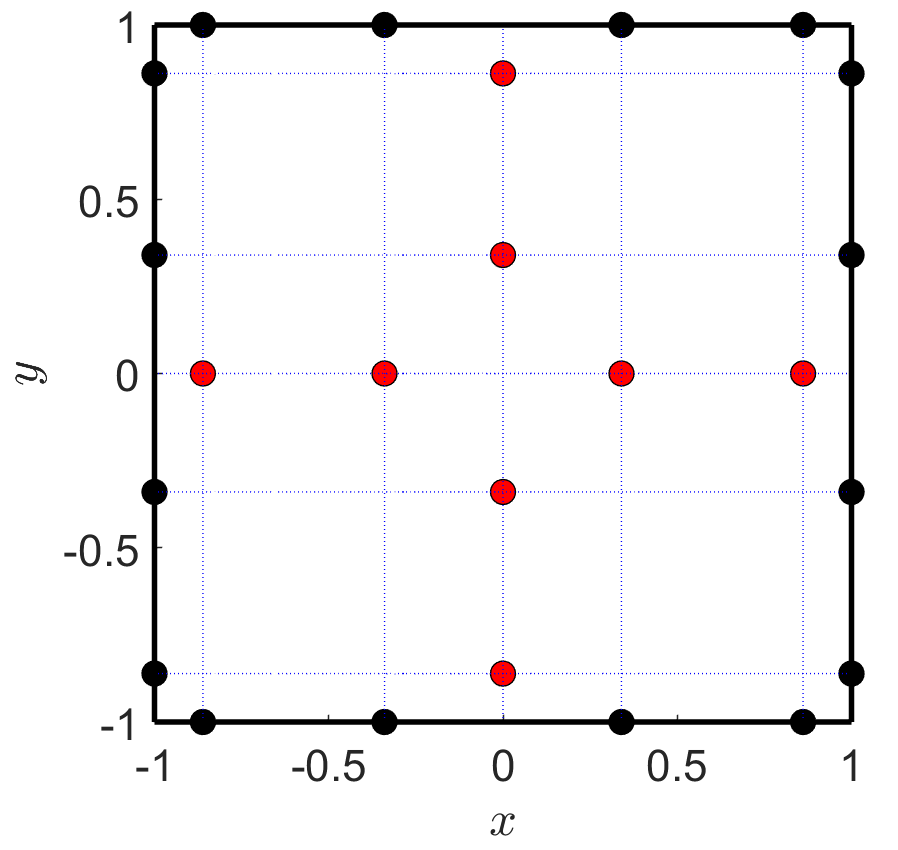}
    \caption{Boundary nodes (black) and classic internal nodes $\mathbb{S}_{ij}^{\tt Zhang-Shu}$ (red) for $\mathbb{P}^3$.}
\end{subfigure}
}
\caption{Nodes of the convex decompositions (\ref{eq:QR2D}) and the classic convex decomposition (\ref{eq:411}) on $\Omega_{ij} = [-1,1]^2$, for the $\mathbb{P}^2$ and $\mathbb{P}^3$ spaces, in the case of $\frac{\dx}{a_1} = \frac{\dy}{a_2}$. }
\label{fig:374} 
\end{figure}

\begin{theorem}\label{thm:OPT2DP2}
	For both $\mathbb{P}^2$ and $\mathbb{P}^3$ spaces, 
	the 2D convex decomposition (\ref{eq:QR2D}) is optimal among all feasible candidates. 
\end{theorem}
\begin{proof}
	It can be easily verified that the 2D convex decomposition (\ref{eq:QR2D}) is feasible for $\mathbb{P}^2$ and $\mathbb{P}^3$. We will prove its optimality by contradiction. Assume that there is another feasible convex decomposition of the form
	\begin{equation}\label{eq:hyp}
		\bar {u}_{ij}^{n} =  
		\sum_{q=1}^Q \omega_{q}^{\tt G}  \left[  
		\omega_1^- u_{i-\frac{1}{2}, q}^{+} +
		\omega_1^+ u_{i+\frac{1}{2}, q}^{-} +
		\omega_2^- u_{q, j-\frac{1}{2}}^{+} +
		\omega_2^+ u_{q, j+\frac{1}{2}}^{-}
		\right]
		+ \sum_{s=1}^S \omega_s p_{ij} ( x_{ij}^{(s)}, y_{ij}^{(s)} ),
	\end{equation}
	which achieves a BP CFL condition milder than (\ref{eq:QR2D}), that is, 	
\begin{equation*}
		c_0 \min \left\{ \frac{\omega_1^- \Delta x}{a_1}, \frac{\omega_1^+ \Delta x}{a_1} , \frac{\omega_2^- \Delta y}{a_2}, \frac{\omega_2^+ \Delta y}{a_2} \right\} 
	>
	c_0 \min \left\{
	\frac{\mu_1 \dx}{2 a_1},
	\frac{\mu_1 \dx}{2 a_1},
	\frac{\mu_2 \dy}{2 a_2},
	\frac{\mu_2 \dy}{2 a_2}
	\right\}
	=
	\frac{c_0}{2 \frac{a_1}{\dx}+2 \frac{a_2}{\dy}+4 \max\left\{ 
		\frac{a_1}{\dx},\frac{a_2}{\dy}\right\}}.
\end{equation*}
	In case of $\frac{a_1}{\dx} \ge \frac{a_2}{\dy}$, we have
	\begin{equation}\label{eq:541}
		\begin{aligned}
			1 & < \min \left\{ 
			\frac{\omega_1^- \Delta x}{a_1}, \frac{\omega_1^+ \Delta x}{a_1} , 
			\frac{\omega_2^- \Delta y}{a_2}, \frac{\omega_2^+ \Delta y}{a_2} 
			\right\} \left[6 \frac{a_1}{\dx}+2 \frac{a_2}{\dy} \right] \\
			& \le 
			3 \frac{\omega_1^- \Delta x}{a_1} \frac{a_1}{\dx}+
			3 \frac{\omega_1^+ \Delta x}{a_1} \frac{a_1}{\dx}+
			\frac{\omega_2^- \Delta y}{a_2} \frac{a_2}{\dy}+
			\frac{\omega_2^+ \Delta y}{a_2} \frac{a_2}{\dy}
			= 3(\omega_1^{+}+\omega_1^{-}) + (\omega_2^{+}+\omega_2^{-}).
		\end{aligned}
	\end{equation}
	In case of $\frac{a_1}{\dx} < \frac{a_2}{\dy}$, we have
	\begin{equation}\label{eq:556}
		\begin{aligned}
			1 & < \min \left\{ 
			\frac{\omega_1^- \Delta x}{a_1}, \frac{\omega_1^+ \Delta x}{a_1} , 
			\frac{\omega_2^- \Delta y}{a_2}, \frac{\omega_2^+ \Delta y}{a_2} 
			\right\} 
			\left[2 \frac{a_1}{\dx}+6 \frac{a_2}{\dy} \right] \\
			& \le
			\frac{\omega_1^- \Delta x}{a_1} \frac{a_1}{\dx}+
			\frac{\omega_1^+ \Delta x}{a_1} \frac{a_1}{\dx}+
			3 \frac{\omega_2^- \Delta y}{a_2} \frac{a_2}{\dy}+
			3 \frac{\omega_2^+ \Delta y}{a_2} \frac{a_2}{\dy}
			=(\omega_1^{+}+\omega_1^{-}) + 3(\omega_2^{+}+\omega_2^{-}).
		\end{aligned}
	\end{equation}
	No matter the hypothetical convex decomposition (\ref{eq:hyp}) is feasible for $\mathbb{P}^2$ or $\mathbb{P}^3$, it should hold exactly for $p_{ij}(x,y)=(x-x_i)^2$ and $p_{ij}(x,y)=(y-y_j)^2$. This gives   
\begin{align*}
	\frac{\dx^2}{12} &= 
(\omega_1^{+}+\omega_1^{-}) \frac{\dx^2}{4} + 
(\omega_2^{+}+\omega_2^{-}) \frac{\dx^2}{12} +
\sum_{s=1}^N \omega_s ({x}_{ij}^{(s)}-x_i)^2,
\\
\frac{\dy^2}{12} &= 
(\omega_1^{+}+\omega_1^{-}) \frac{\dy^2}{12} + 
(\omega_2^{+}+\omega_2^{-}) \frac{\dy^2}{4} +
\sum_{s=1}^N \omega_s ({y}_{ij}^{(s)}-y_j)^2,
\end{align*}
	implying  
	\begin{equation}\label{eq:553}
		3(\omega_1^{+}+\omega_1^{-}) + (\omega_2^{+}+\omega_2^{-}) \le 1
		\quad \text{and} \quad
		(\omega_1^{+}+\omega_1^{-}) + 3(\omega_2^{+}+\omega_2^{-}) \le 1,
	\end{equation}
	 which contradict with either (\ref{eq:541}) or (\ref{eq:556}). Hence the assumption is incorrect, and decomposition (\ref{eq:QR2D}) is optimal.
\end{proof}

\begin{remark1}
	The standard CFL condition for linear stability of the $\mathbb P^k$-based DG method with a $(k+1)$-stage $(k+1)$-order Runge--Kutta (RK) time discretization \cite{cockburn2001runge} is given by the following empirical formula   
	\begin{equation}\label{eq:LS}
		\left(   \frac{a_1}{\dx}+ \frac{a_2}{\dy} \right) \Delta t \le \frac{1}{2k+1}.  
	\end{equation} 
	Table \ref{tab:specase} gives a comparison of different CFL conditions in the special case of $\frac{\dx}{a_1} = \frac{\dy}{a_2} = h$ for the $\mathbb P^2$-based (third-order) and  $\mathbb P^3$-based (fourth-order) DG methods. 
	 One can see that if $c_0=1$, 
	 the optimal BP CFL condition \eqref{eq:CFL2D} of the DG schemes (with the BP limiter) is even weaker than the standard one \eqref{eq:LS}. 


  	\begin{table}[htbp] 
  		\centering
  		\caption{Comparison of different CFL conditions in the special case of $\frac{\dx}{a_1} = \frac{\dy}{a_2} = h$.  }
  		\label{tab:specase}
  		\renewcommand\arraystretch{1.6}
  		\setlength{\tabcolsep}{6mm}{
  			\begin{tabular}{cccccc}
  				\toprule[1.5pt]
  				
  				\multirow{2}{*}{} & 
  				\multirow{2}{*}{linear stability} &
  				\multicolumn{2}{c}{ $c_0 = 1$ } & 
  				\multicolumn{2}{c}{ $c_0 = {1}/{2}$} \\
  				\cmidrule(r){3-4} \cmidrule(l){5-6} 
  				& & 
  				optimal & classic &
  				optimal & classic \\
  				
  				\midrule[1.5pt]
  				
  				$\mathbb{P}^2$ & $\dt \leq \frac{1}{10}h$ &  
  					
  				\multirow{2}{*}{$\dt \leq \frac{1}{8}h$} &
  				\multirow{2}{*}{$\dt \leq \frac{1}{12}h$} &
  				\multirow{2}{*}{$\dt \leq \frac{1}{16}h$} &
  				\multirow{2}{*}{$\dt \leq \frac{1}{24}h$} \\
 				
  				$\mathbb{P}^3$ & $\dt \leq \frac{1}{14}h$ &  & & & \\
  				
  				\bottomrule[1.5pt]
  			\end{tabular}
  		}
  	\end{table}
  	
\end{remark1}

\begin{remark1}
We would like to clarify that the optimal BP CFL condition \eqref{eq:CFL2D} is merely the best among all those achieved via feasible convex decomposition. It does not mean that such an optimal condition is always sharp or necessary, as other possible analysis approaches may give perhaps weaker BP conditions.
\end{remark1}

\section{Optimal 3D convex decomposition for $\mathbb{P}^\textnormal{2}$ and $\mathbb{P}^\textnormal{3}$ on cuboid cells}

The proposed 2D optimal convex decomposition (\ref{eq:QR2D}) can be extended to 3D and higher dimensions. We denote the maximum characteristic speeds in the $x$-, $y$- and $z$-directions by $a_1$, $a_2$, and $a_3$, respectively. 
Let $\Omega_{ij\ell}=[x_{i-\frac12},x_{i+\frac12}]\times[y_{j-\frac12},y_{j+\frac12}] \times [z_{\ell-\frac12},z_{\ell+\frac12}]$ be a cuboid cell, with $\Delta x$, $\Delta y$, and $\Delta z$ denoting its lengths in the $x$-, $y$- and $z$-directions, respectively.

\begin{decomp}{Optimal 3D convex decomposition} 
For $p_{ij\ell} \in \mathbb P^2$ or $\mathbb P^3$, the {\em optimal} 3D convex decomposition on $\Omega_{ij\ell}$ is given by
\begin{equation}\label{eq:QR3D}
\begin{aligned}
\bar u_{ij\ell}
& =
\frac{\mu_1}{2} \sum_{q=1}^{Q}\sum_{r=1}^{Q} \omega_{q}^{\tt G} \omega_{r}^{\tt G} 
\left[
  u_{i-\frac12,q,r}^{+} +
  u_{i+\frac12,q,r}^{-} 
\right] +
\frac{\mu_2}{2} \sum_{q=1}^{Q}\sum_{r=1}^{Q} \omega_{q}^{\tt G} \omega_{r}^{\tt G} 
\left[
  u_{r,j-\frac12,q}^{+} +
  u_{r,j+\frac12,q}^{-}
\right]  \\
& +
\frac{\mu_3}{2} \sum_{q=1}^{Q}\sum_{r=1}^{Q} \omega_{q}^{\tt G} \omega_{r}^{\tt G} 
\left[
  u_{q,r,\ell-\frac12}^{+} +
  u_{q,r,\ell+\frac12}^{-}
\right] + 
\frac{\omega}{2}\sum_{s} p_{ij\ell}\left( \hat{x}_s,\hat{y}_s,\hat{z}_s \right),
\end{aligned}
\end{equation}
where $\bar u_{ij\ell}$ denotes the cell average of $p_{ij\ell}$ over $\Omega_{ij\ell}$, and  
\begin{equation*}
	u_{i\pm\frac12,q,r}^{\mp} = p_{ij\ell} \big( x_{i\pm\frac12}  , y_{j,q}^{\tt G}, z_{\ell,r}^{\tt G} \big), \quad
u_{r,j\pm\frac12,q}^{\mp} = p_{ij\ell} \big( x_{i,r}^{\tt G}, y_{j\pm\frac12}  , z_{\ell,q}^{\tt G} \big), \quad 
u_{q,r,k\pm\frac12}^{\mp} = p_{ij\ell} \big( x_{i,q}^{\tt G}, y_{j,r}^{\tt G}  , z_{\ell\pm\frac12}   \big).
\end{equation*}
In \eqref{eq:QR3D}, the weights $\mu_1$, $\mu_2$, $\mu_3$, and $\omega$ are given by
\begin{equation*}
	\mu_1 = \frac{\phi_1}{\psi}, \quad
\mu_2 = \frac{\phi_2}{\psi}, \quad
\mu_3 = \frac{\phi_3}{\psi}, \quad
\omega = \frac{\phi_*}{\psi}
\end{equation*}
with 
\begin{equation*}
	\phi_1 = a_1/\dx, \quad
\phi_2 = a_2/\dy, \quad
\phi_3 = a_3/\dz, \quad
\psi = \phi_1+\phi_2+\phi_3+2\phi_*, \quad
\phi_* = \max \{\phi_1,\phi_2,\phi_3\}, 
\end{equation*}
and the internal nodes are given by
\begin{equation*}
\mathbb{S}_{ij\ell}^{\tt optimal} =	\left\{ \big( \hat{x}_{s},\hat{y}_{s},\hat{z}_{s} \big) \right\} = 
\begin{dcases}
~ \left(
	x_i, 
	y_j \pm \frac{\dy}{\sqrt{6}} \sqrt{\frac{\phi_*-\phi_2}{\phi_*}} ,
	z_\ell 
\right) 
\; \text{and} \;
\left(
	x_i,
	y_j,
	z_\ell \pm \frac{\dz}{\sqrt{6}} \sqrt{\frac{\phi_*-\phi_3}{\phi_*}}
\right),
& \text{if} \; \phi_1 = \max \{ \phi_1,\phi_2,\phi_3 \},\\
~
\left( 
	x_i , 
	y_j ,
	z_\ell \pm \frac{\dz}{\sqrt{6}} \sqrt{\frac{\phi_*-\phi_3}{\phi_*}}
\right)
\; \text{and} \;
\left(
	x_i \pm \frac{\dx}{\sqrt{6}} \sqrt{\frac{\phi_*-\phi_1}{\phi_*}},
	y_j,
	z_\ell 
\right),
& \text{if} \; \phi_2 = \max\{\phi_1,\phi_2,\phi_3 \},\\
~ \left(
	x_i \pm \frac{\dx}{\sqrt{6}} \sqrt{\frac{\phi_*-\phi_1}{\phi_*}},
	y_j,
	z_\ell 
\right)
\; \text{and} \;
\left( 
	x_i , 
	y_j \pm \frac{\dy}{\sqrt{6}} \sqrt{\frac{\phi_*-\phi_2}{\phi_*}},
	z_\ell  
\right),
& \text{if} \; \phi_3 = \max\{\phi_1,\phi_2,\phi_3 \}.
\end{dcases}
\end{equation*}
\end{decomp}

\begin{theorem}\label{thm:OPT3DP2}
	For both $\mathbb{P}^2$ and $\mathbb{P}^3$ spaces, 
the 3D convex decomposition (\ref{eq:QR3D}) is optimal among all feasible candidates. 	
\end{theorem}

The proof of Theorem \ref{thm:OPT3DP2} is similar to that of Theorem \ref{thm:OPT2DP2} and is thus omitted. As 
summarized in Table \ref{tab:530}, 
the advantages of the optimal 3D convex decomposition (\ref{eq:QR3D}), over the 3D versions of the classic decompositions, are even greater than the 2D case. 
Figure \ref{fig:627} shows 
the boundary and internal nodes for further comparison. 


\begin{table}[h]
\centering
\caption{BP CFL conditions and the internal node sets of the optimal 3D convex decomposition (\ref{eq:QR3D}) and the 3D versions of the Zhang--Shu and Jiang--Liu convex decompositions for the $\mathbb{P}^2$ and $\mathbb{P}^3$ spaces.}\label{tab:530}
\renewcommand\arraystretch{1.7}
\begin{tabular}{ lrrrr} 
 \toprule[1.5pt]
  & BP CFL condition & BP CFL condition & Internal  & Internal  \\
 & general case & $\frac{\dx}{a_1} = \frac{\dy}{a_2} = \frac{\dz}{a_3} = h$ &  $\mathbb{P}^2$ nodes & $\mathbb{P}^3$ nodes  \\ 
 
 \midrule[1.5pt]
 
Optimal & 
$\left[2\frac{a_1}{\dx}+2\frac{a_2}{\dy}+2\frac{a_3}{\dz}+4\max\left\{ 
\frac{a_1}{\dx},
\frac{a_2}{\dy},
\frac{a_3}{\dz}
\right\} \right] \Delta t \le c_0$ 
& $\Delta t \le \frac{c_0}{10} h$ & 1 $\sim$ 4 & 1 $\sim$ 4 \\ 
 Zhang \& Shu \cite{zhang2010} & $\left[6\frac{a_1}{\dx}+6\frac{a_2}{\dy}+6\frac{a_3}{\dz}\right]\Delta t \le c_0$ & $\Delta t \le \frac{c_0}{18} h$ & 19 & 48 \\ 
 Jiang \& Liu \cite{jiang2018invariant} & $\left[18 \max\left\{ 
\frac{a_1}{\dx},
\frac{a_2}{\dy},
\frac{a_3}{\dz}
\right\} \right]\Delta t \le c_0$ & $\Delta t \le \frac{c_0}{18} h$ & 19 & 48\\ 

\bottomrule[1.5pt]
\end{tabular}
\end{table}

\begin{figure}[htb]
	\centering
	\begin{subfigure}{0.24\textwidth}
		\includegraphics[width=\textwidth]{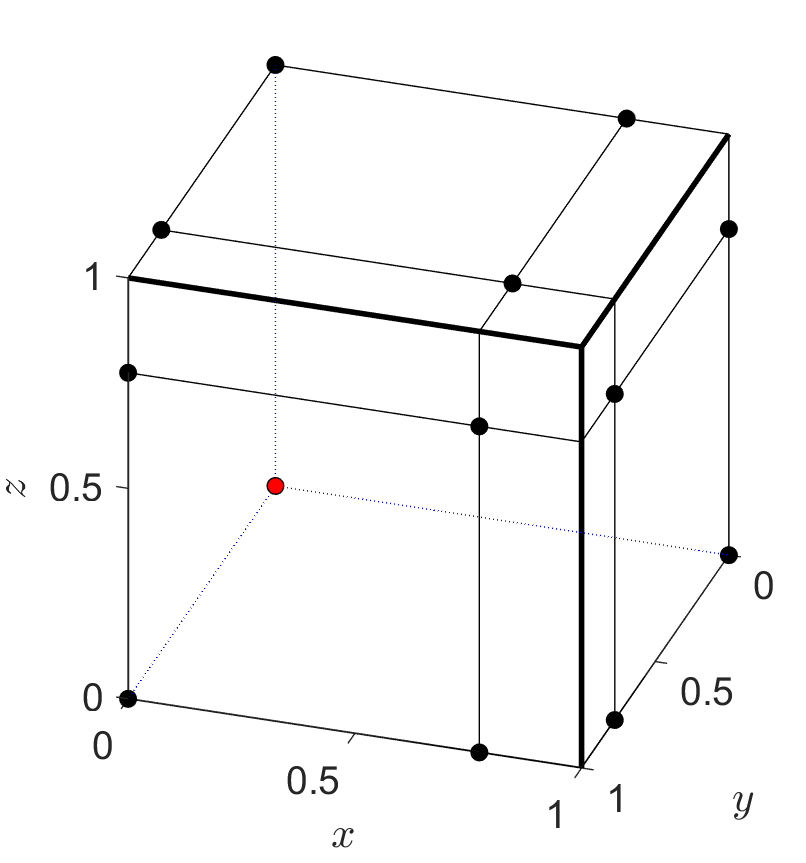}
		\caption{Boundary nodes (black) and optimal internal nodes $\mathbb{S}_{ij\ell}^{\tt optimal}$ (red) for $\mathbb{P}^2$.}
	\end{subfigure}
	\hfill
	\begin{subfigure}{0.24\textwidth}
		\includegraphics[width=\textwidth]{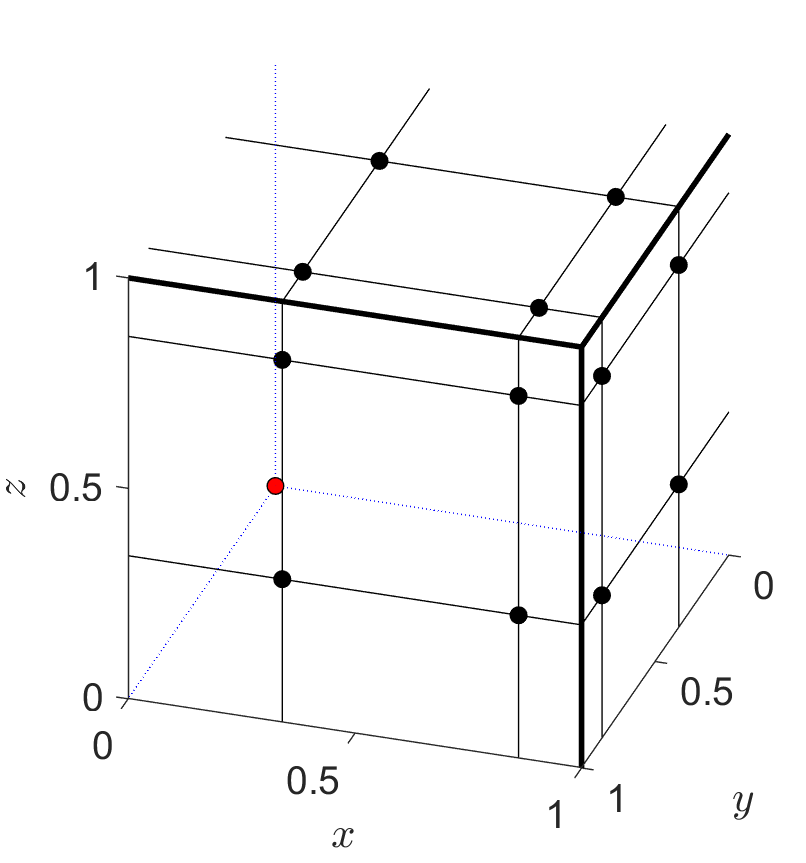}
		\caption{Boundary nodes (black) and optimal internal nodes $\mathbb{S}_{ij\ell}^{\tt optimal}$ (red) for $\mathbb{P}^3$.}
	\end{subfigure}
	\hfill
	\begin{subfigure}{0.24\textwidth}
		\includegraphics[width=\textwidth]{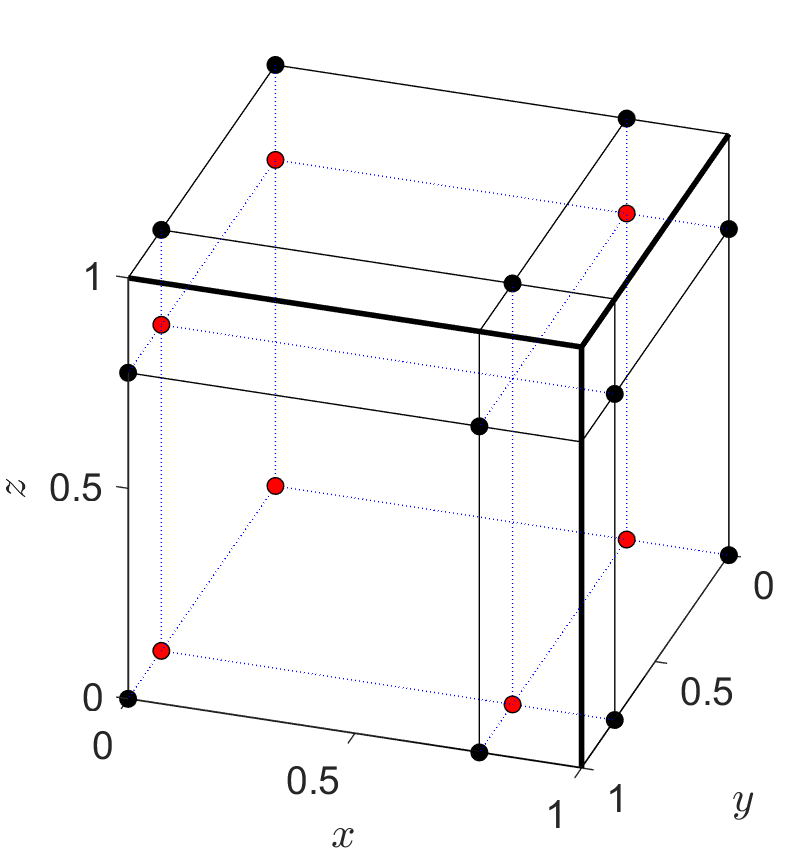}
		\caption{Boundary nodes (black) and classic internal nodes $\mathbb{S}_{ij\ell}^{\tt Zhang-Shu}$ (red) for $\mathbb{P}^2$.}
	\end{subfigure}
	\hfill
	\begin{subfigure}{0.24\textwidth}
		\includegraphics[width=\textwidth]{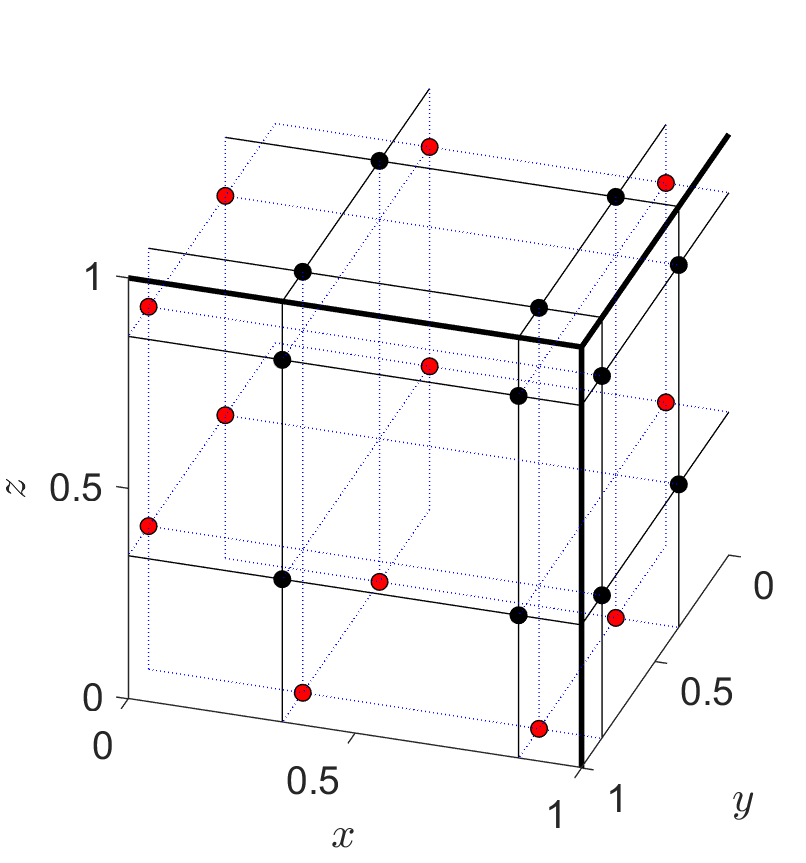}
		\caption{Boundary nodes (black) and classic internal nodes $\mathbb{S}_{ij\ell}^{\tt Zhang-Shu}$ (red) for $\mathbb{P}^3$.}
	\end{subfigure}
	\caption{Nodes of the optimal 3D decomposition (\ref{eq:QR3D}) and the 3D version of the classic Zhang--Shu decomposition on $\Omega_{ij\ell} = [-1,1]^3$, for $\mathbb{P}^2$ and $\mathbb{P}^3$, in the case of $\frac{\dx}{a_1} = \frac{\dy}{a_2} = \frac{\dz}{a_3}$. 
		We only plot the nodes in the first octant $[0,1]^3$, while the other nodes are distributed symmetrically.} 
	\label{fig:627}
\end{figure}

\begin{figure}[h]
	\centering
	\centerline{
		\hfill
		\includegraphics[width=0.32\textwidth]{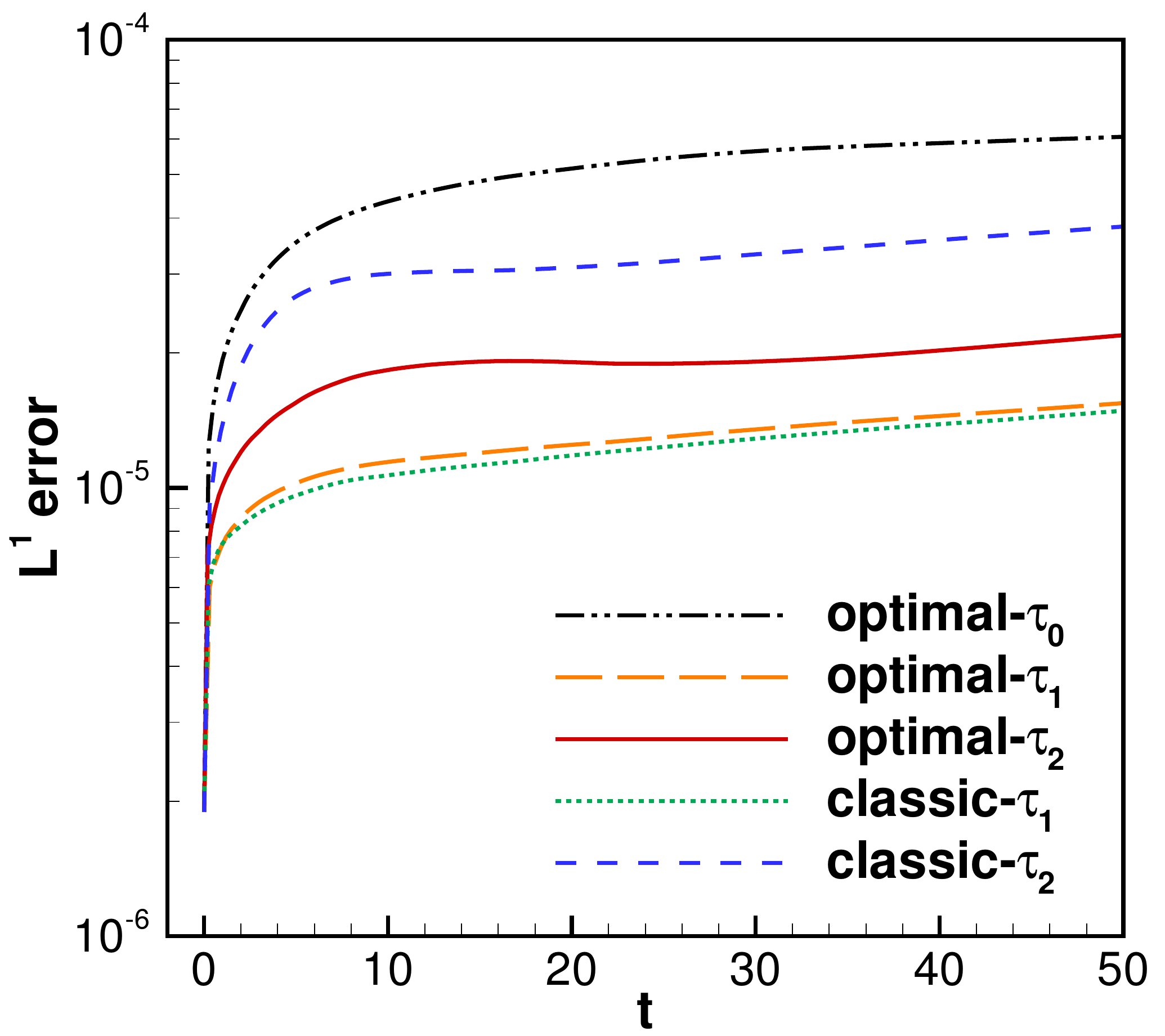}
		\hfill
		\includegraphics[width=0.32\textwidth]{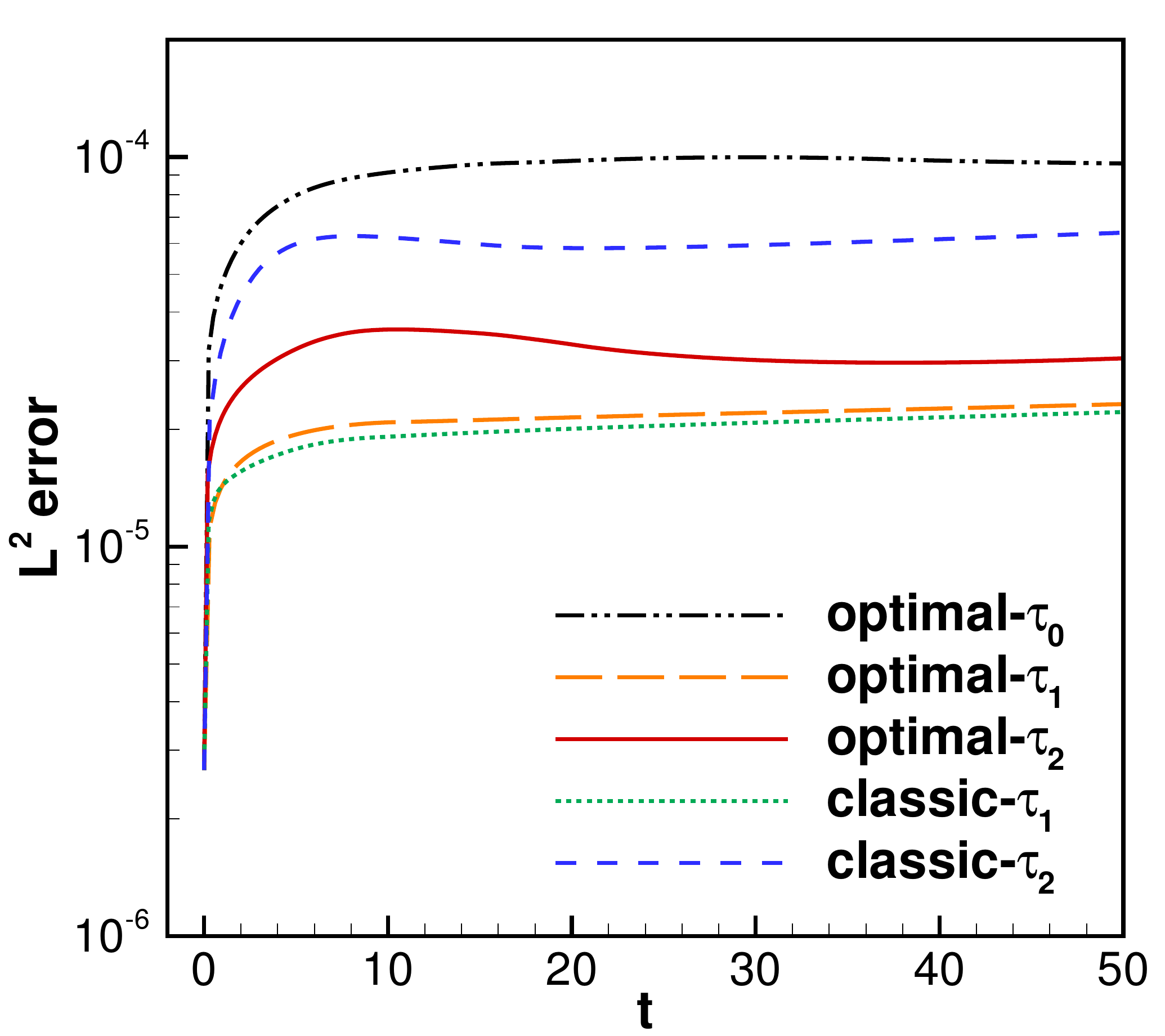}
		\hfill
		\includegraphics[width=0.32\textwidth]{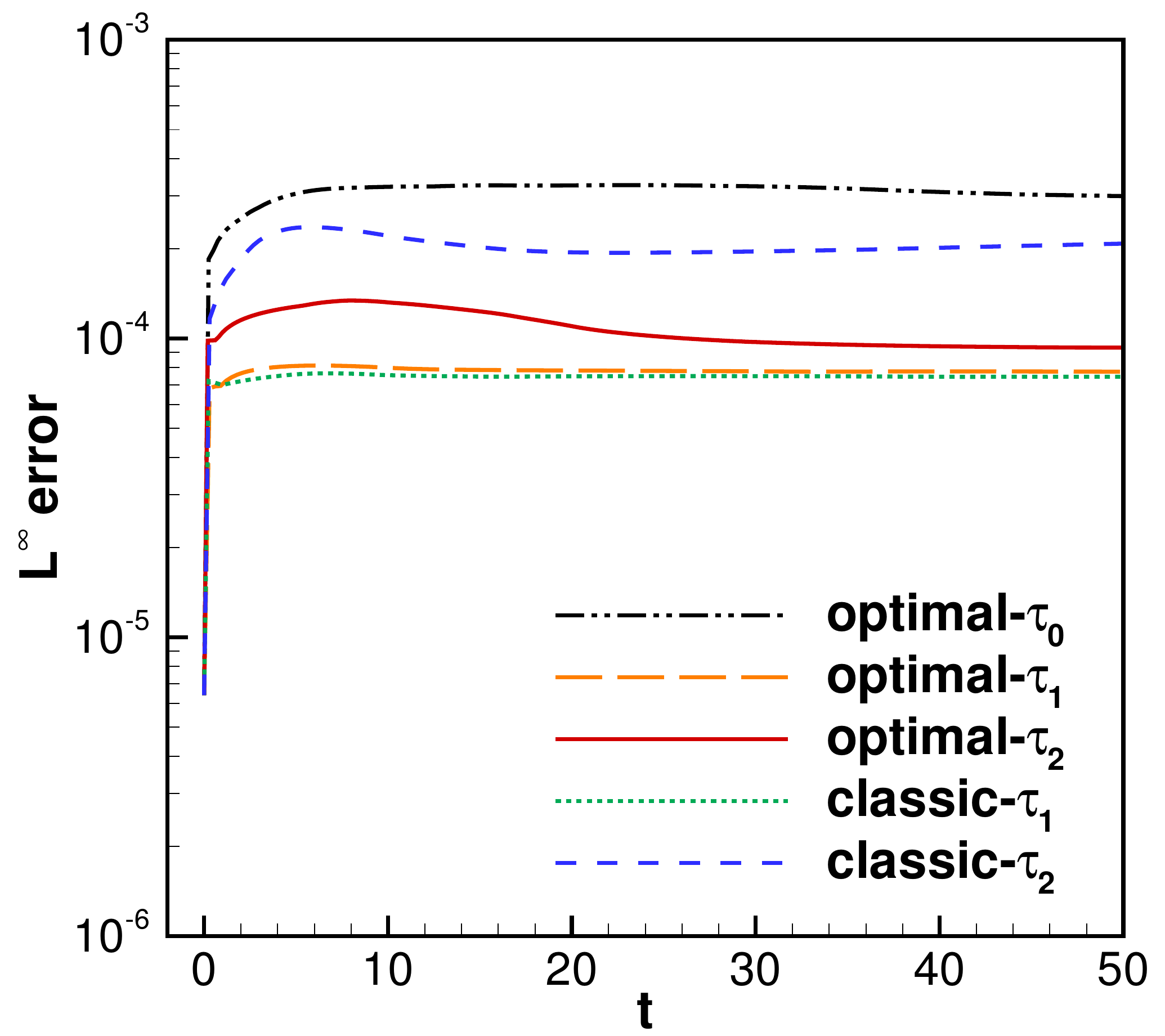}
		\hfill
	}
	%
	\centerline{
		\hfill
		\includegraphics[width=0.32\textwidth]{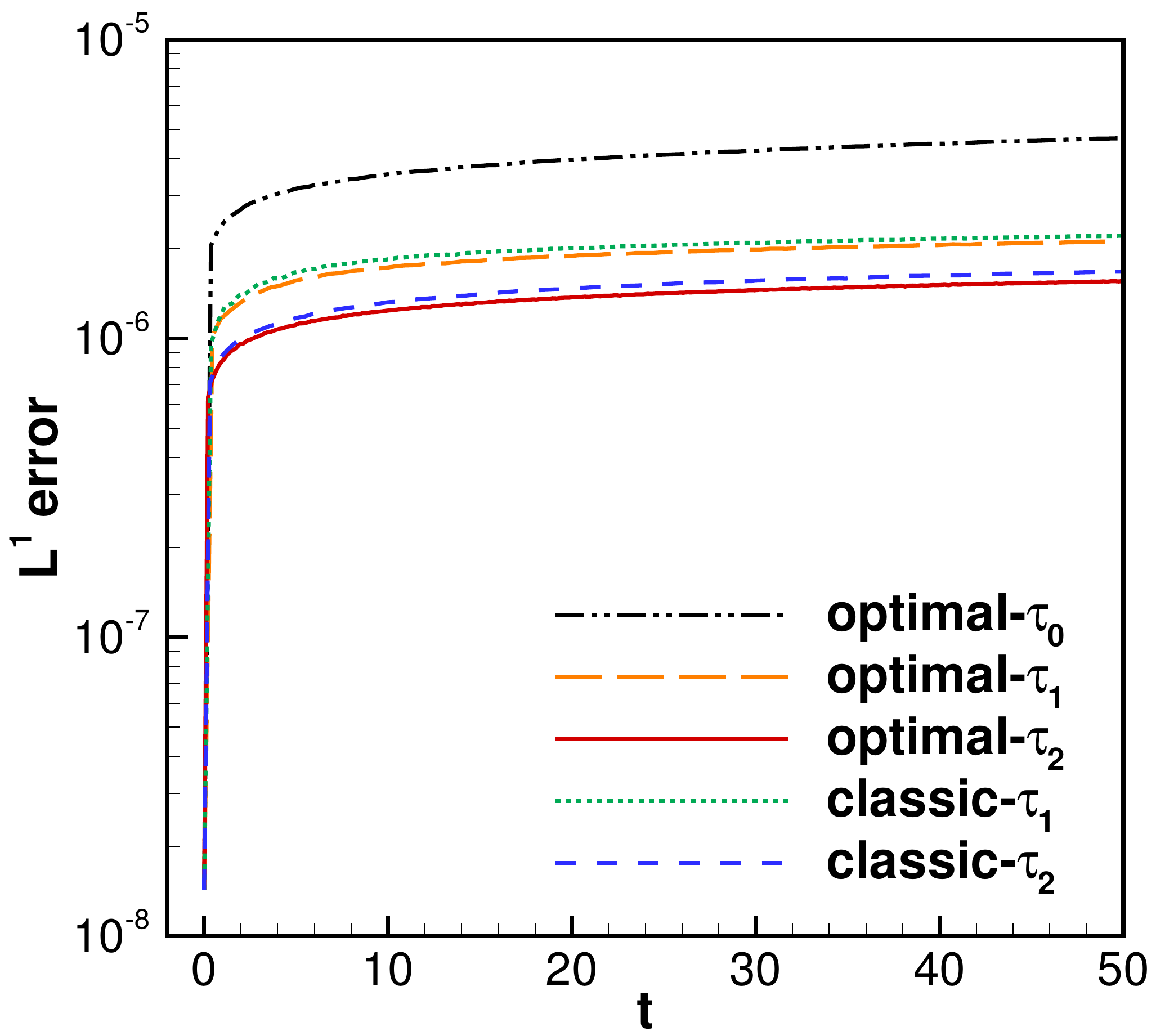}
		\hfill
		\includegraphics[width=0.32\textwidth]{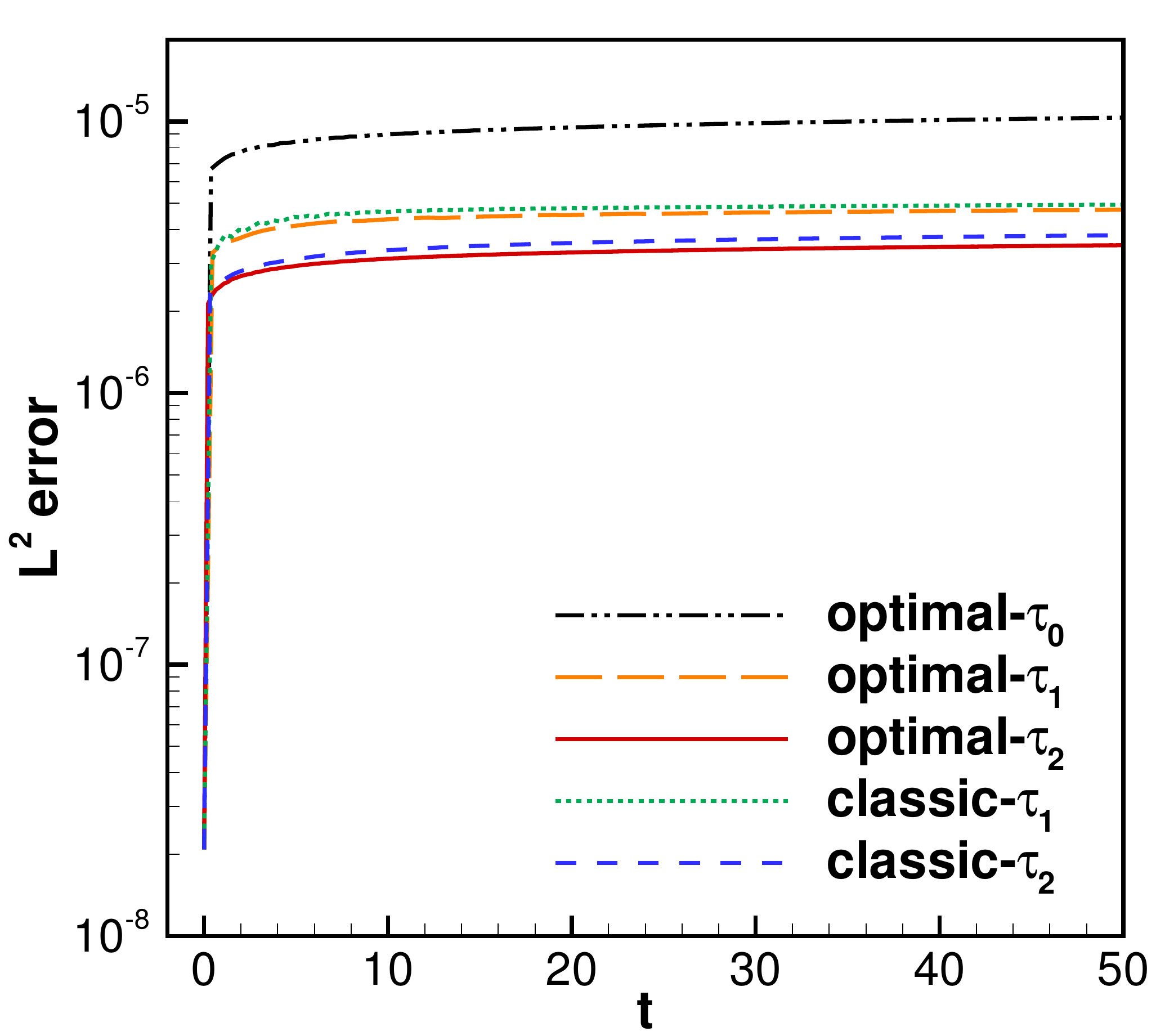}
		\hfill
		\includegraphics[width=0.32\textwidth]{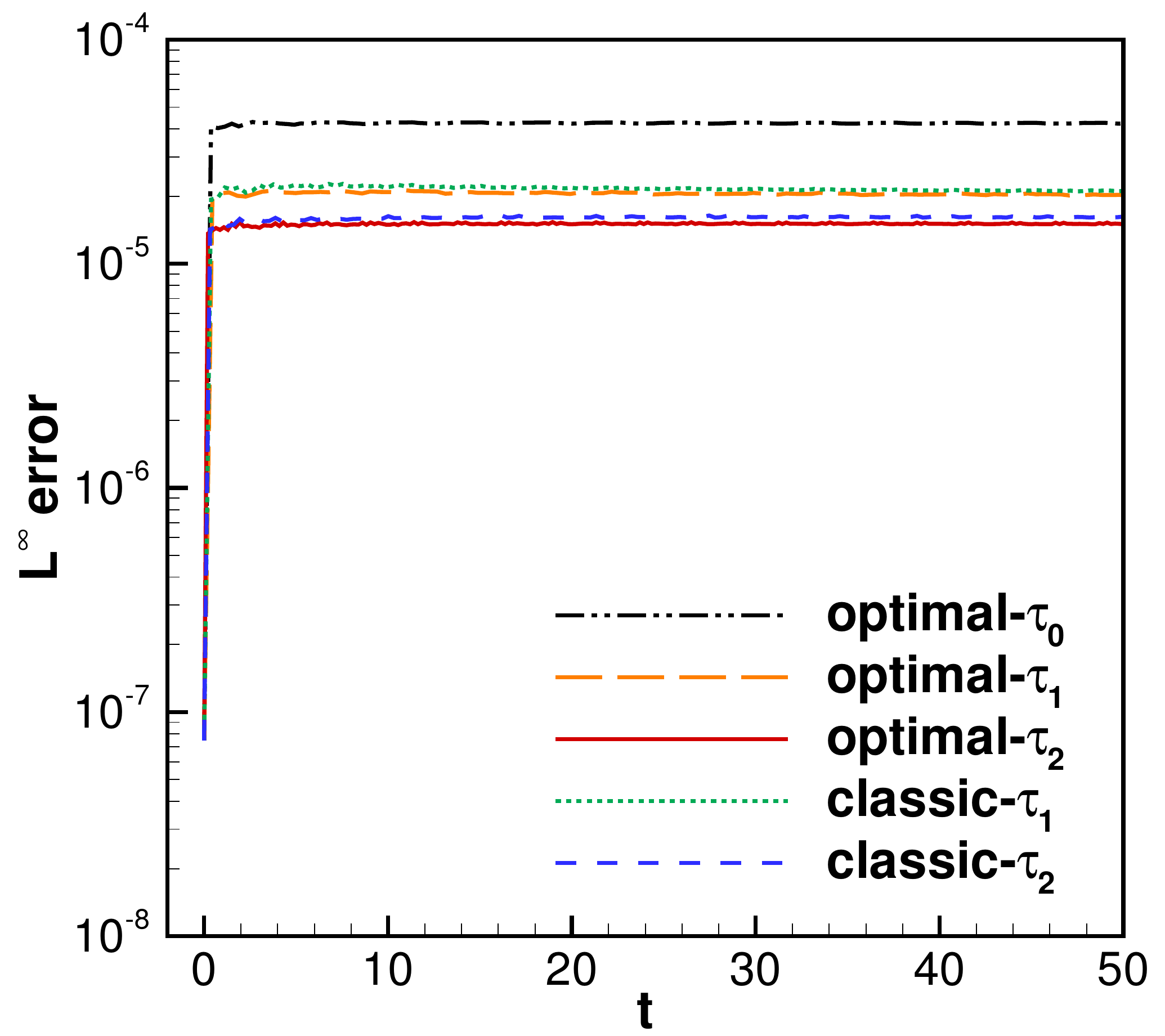}
		\hfill
	}
	\caption{Example \ref{Ex1}: Time evolution of the $L^1$, $L^2$, and $L^{\infty}$ errors in the numerical solutions obtained using the third-order ($\mathbb{P}^2$; top row) and fourth-order ($\mathbb{P}^3$; bottom row) BP DG schemes designed via different approaches and by using different time step-sizes. 
	}
	\label{fig:ex1_2}
\end{figure}

\section{Numerical experiments}

In this section, we test the accuracy, robustness, and efficiency of the high-order BP DG schemes designed via the proposed optimal convex decomposition \eqref{eq:QR2D}, which are referred to as the ``{\tt optimal} approach'' for short. For comparison, we also consider the high-order BP DG schemes designed via the classic Zhang--Shu convex decomposition \eqref{eq:411}, which are referred to as the ``{\tt classic} approach'' for short. 
While the convex decomposition is independent of the choice of BP numerical fluxes, we adopt 
the global Lax--Friedrichs flux with $c_0=1$ in all the presented tests. 
Unless otherwise stated, we employ the three-stage third-order SSP RK discretization  \cite{GottliebKetchesonShu2011} (abbreviated as {\tt SSPRK3}) for the $\mathbb P^2$-based (third-order) DG scheme, 
and  the five-stage fourth-order SSP RK time discretization \cite{GottliebKetchesonShu2011} (abbreviated as {\tt SSPRK4})
 for the $\mathbb P^3$-based (fourth-order) DG scheme. 
The SSP coefficients for {\tt SSPRK3} and {\tt SSPRK4} are $C_{\tt SSP}=1$ and $C_{\tt SSP} \approx 1.508$, respectively. 
All  the  methods  are  implemented using C++ language with double precision on a Linux server with 
Intel(R) Xeon(R) Platinum 8268 CPU @ 2.90GHz 2TB RAM.



\begin{example}[Linear convection equation]\label{Ex1}
	We start with the two-dimensional linear convection equation 
	\begin{equation}\label{eq:791}
		u_t+u_x + u_y=0, \qquad (x,y,t)\in[-1,1]\times[-1,1]\times\mathbb{R}^+,
	\end{equation}
	with periodic boundary conditions and the initial data $u(x,y,0) = \sin(\pi(x+y))$. 
	The exact solution satisfies a maximum principle, implying the convex invariant region $G=[-1,1]$. 
	We simulate this problem until $t=50$ to study the long-time stability of the BP DG schemes with the BP limiter on the  uniform mesh of $100 \times 100$ cells.     
	Figure \ref{fig:ex1_2} shows the time evolution of the numerical errors in the 
	$L^1$, $L^2$, and $L^{\infty}$ norms, respectively. In the simulations, we adopt three different time step-sizes, 
	namely, $\tau_0 :=  C_{\tt SSP}  \tau^{\tt BP}_{\tt O}$ with $\tau^{\tt BP}_{\tt O}$ being 
	the maximum $\Delta t$ determined by the optimal BP CFL condition \eqref{eq:CFL2D}, 
	$\tau_1 :=  C_{\tt SSP}  \tau^{\tt BP}_{\tt C}$ with $\tau^{\tt BP}_{\tt C}$ being 
	the maximum $\Delta t$ determined by the classic Zhang--Shu BP CFL condition \eqref{eq:CFL-ZS}, 
	and $\tau_2 :=  C_{\tt SSP}  \tau^{\tt LS}$ with $\tau^{\tt LS}$ being the 
	maximum $\Delta t$ determined by the standard CFL condition \eqref{eq:LS} for linear stability. 
	The CPU time of all these simulations is presented in Table \ref{tab:CPUEx1}. 
	One can see that all the errors shown in Figure \ref{fig:ex1_2} do not exhibit any sign of exponential growth, which indicates the simulations are stable. 
	We also observe that, when the identical time step is used,  
	the errors of the {\tt optimal} approach is smaller or comparable to those of the {\tt classic} one, 
	but the {\tt optimal} approach uses less CPU time (see Table \ref{tab:CPUEx1}) due to the fewer internal nodes. 
	It is seen that the numerical errors are slightly larger, if a bigger time step is adopted, as expected. 
	The {\tt optimal} approach allows a larger time step, with which the CPU time is much less, as shown in Table \ref{tab:CPUEx1}.

  	\begin{table}[htbp] 
	\centering
	\caption{CPU time in seconds for Example \ref{Ex1}.}
		\label{tab:CPUEx1}
	\renewcommand\arraystretch{1.5}
	\setlength{\tabcolsep}{5mm}{
		\begin{tabular}{cccccc}
			\toprule[1.5pt]
			
			\multirow{2}{*}{} & 
			\multicolumn{3}{c}{{\tt optimal} approach} & 
			\multicolumn{2}{c}{{\tt classic} approach} \\
			\cmidrule(r){2-4} \cmidrule(l){5-6} 
			&
			$\tau_0$ & $\tau_1$ & $\tau_2$ &
			$\tau_1$ & $\tau_2$  \\
			
			\midrule[1.5pt]
			
			$\mathbb{P}^2$ & 1563.22 & 2410.99 & 1980.90 & 2603.67 & 2166.29 \\
			
			$\mathbb{P}^3$ & 3058.73 & 4564.84 & 5330.88 & 4947.13 & 5775.08 \\
			
			\bottomrule[1.5pt]
		\end{tabular}
	}
\end{table}

\end{example}

\begin{example}[Riemann problem of Burgers' equation]\label{Ex3}
This example \cite{lu2021oscillation} considers	the inviscid Burgers' equation 
	\begin{align}
		u_t+\bigg(\frac{u^2}{2}\bigg)_x + \bigg(\frac{u^2}{2}\bigg)_y=0, \qquad (x,y,t)\in[0,1]\times[0,1]\times\mathbb{R}^+,
	\end{align}
	with outflow boundary conditions and the initial condition 
\begin{align*}
	u(x,y,0) = 
	\begin{cases} 
		-0.2, & \text{ if } x < 0.5, y> 0.5, \\
		-1, & \text{ if } x > 0.5, y> 0.5, \\
		0.5, & \text{ if } x < 0.5, y < 0.5, \\
		0.8, & \text{ if } x > 0.5, y < 0.5.  
	\end{cases}
\end{align*}
The exact solution  of this example also obeys the maximum principle with $G=[-1,0.8]$.  
Figure \ref{fig:ex3_1} displays the numerical results at $t=0.5$ computed with $256\times 256$ uniform cells. 
We observe that the {\tt optimal} approach with time step-size $\tau_0=C_{\tt SSP}  \tau^{\tt BP}_{\tt O}$ and the {\tt classic} approach with time step-size $\tau_1=C_{\tt SSP}  \tau^{\tt BP}_{\tt C}$ give very similar results, and 
the discontinuities are equally well resolved by both approaches. 
However, the CPU time of the {\tt optimal} approach is much less than that of the 
{\tt classic} approach, as exhibited in Table \ref{tab:ex-cpu}.


\begin{figure}[h]
	\centering
	\begin{subfigure}{0.31\textwidth}
		\includegraphics[width=\textwidth]{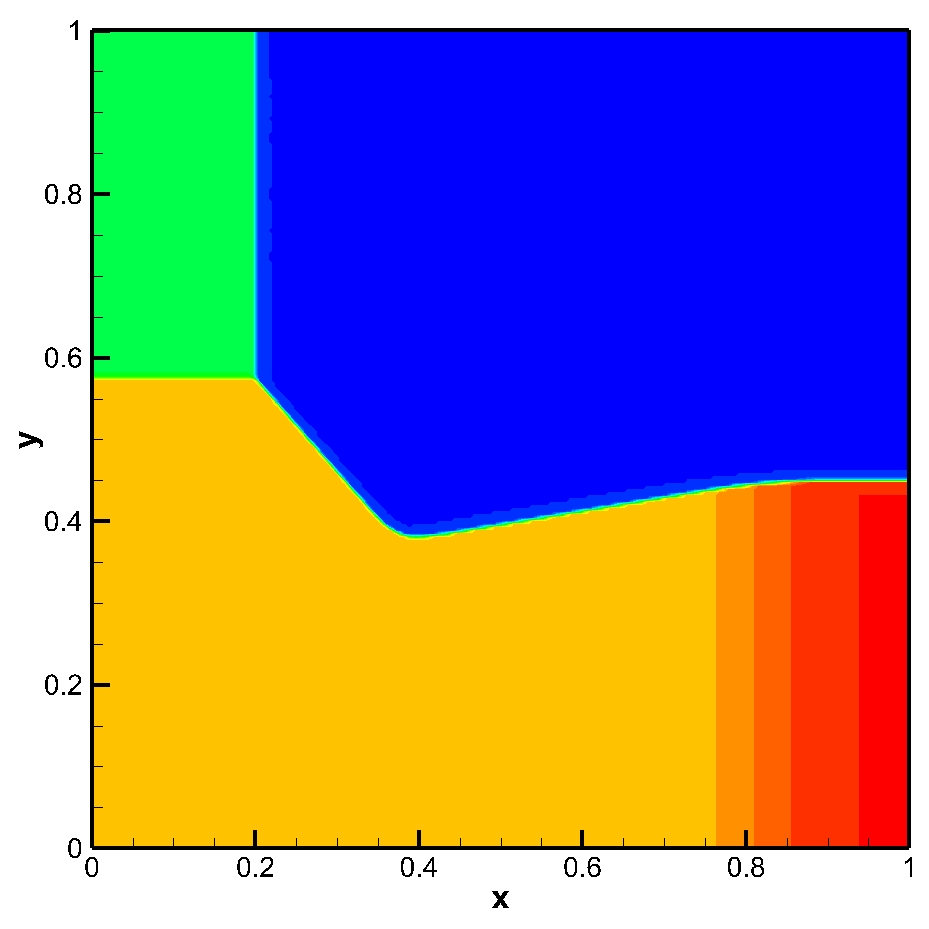}
		\caption{ $\mathbb{P}^2$: {\tt optimal} approach.}
	\end{subfigure}
	\hfill
	\begin{subfigure}{0.31\textwidth}
		\includegraphics[width=\textwidth]{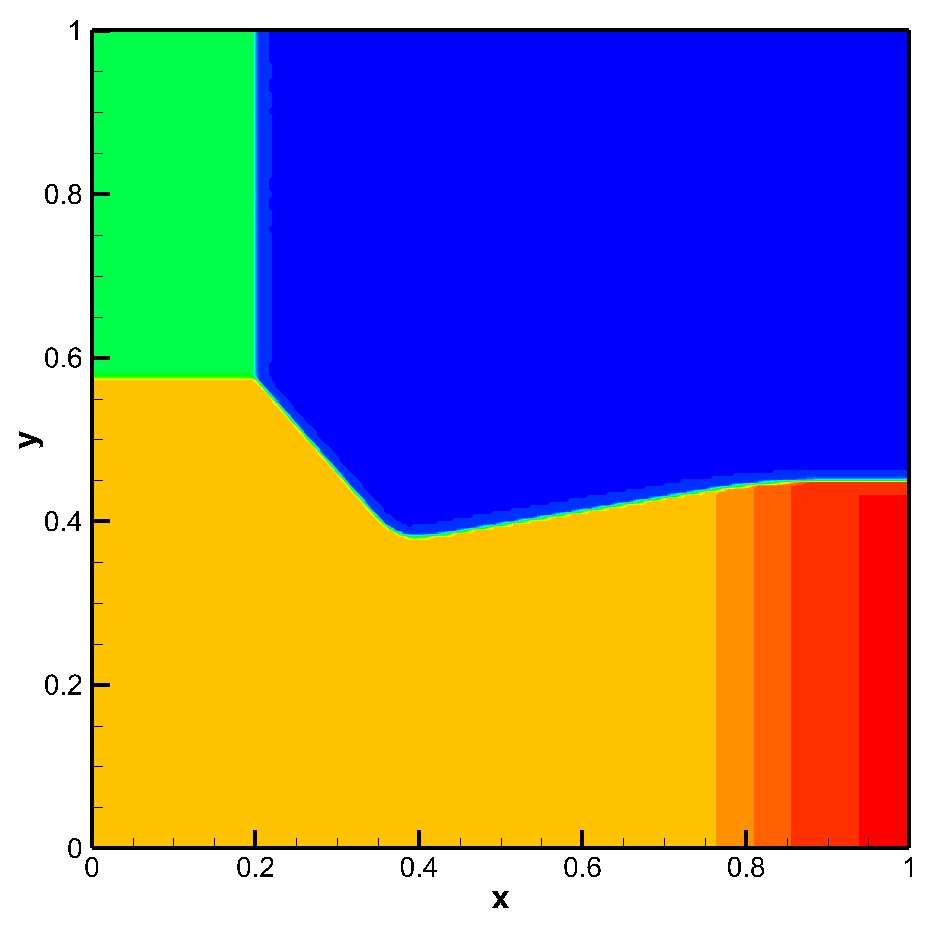}
		\caption{ $\mathbb{P}^2$: {\tt classic} approach.}
	\end{subfigure}
	\hfill
	\begin{subfigure}{0.35\textwidth}
		\includegraphics[width=\textwidth]{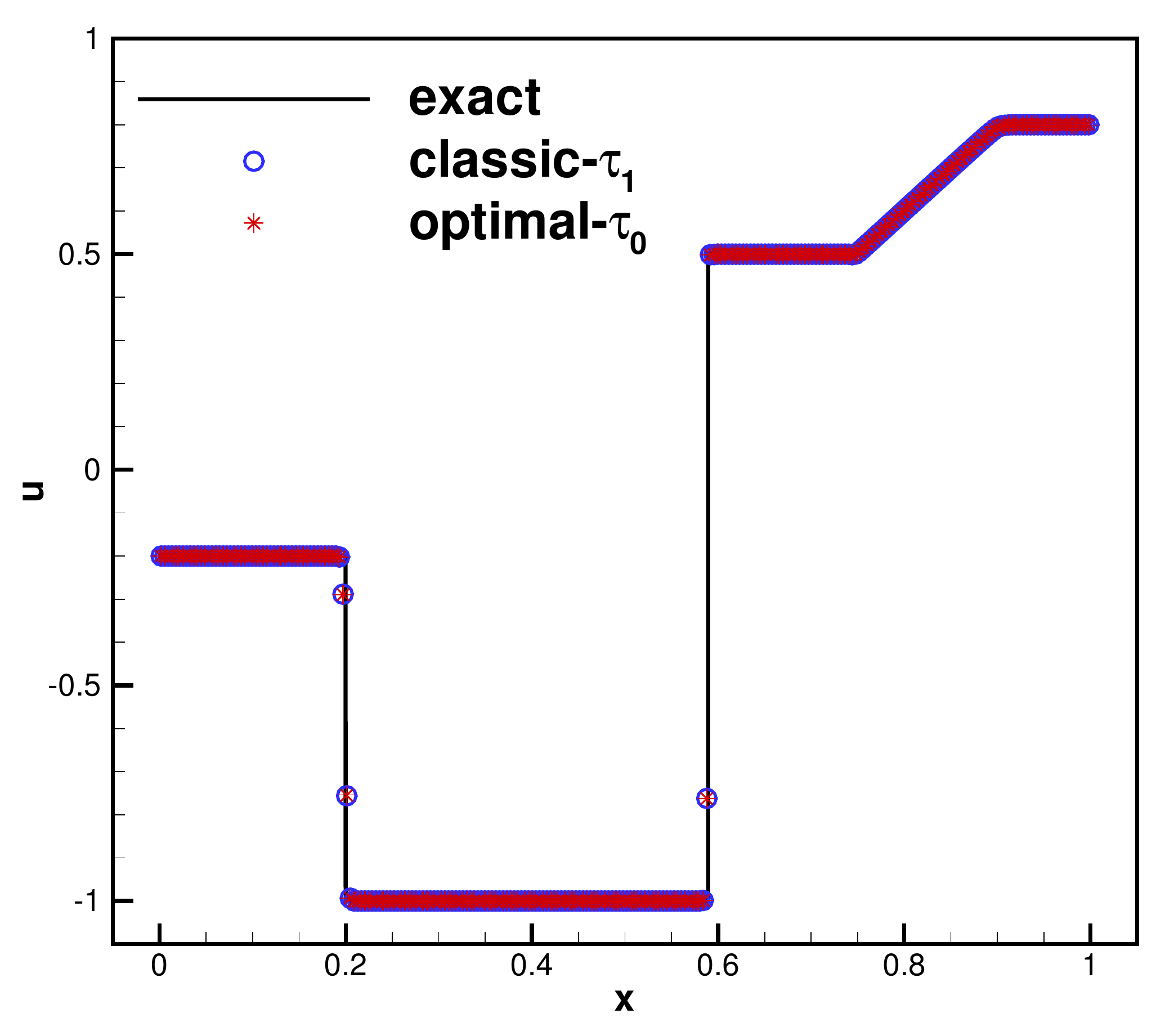}
		\caption{ $\mathbb{P}^2$: cut along the diagonal line  $x=1-y$.}
	\end{subfigure}

	\begin{subfigure}{0.31\textwidth}
		\includegraphics[width=\textwidth]{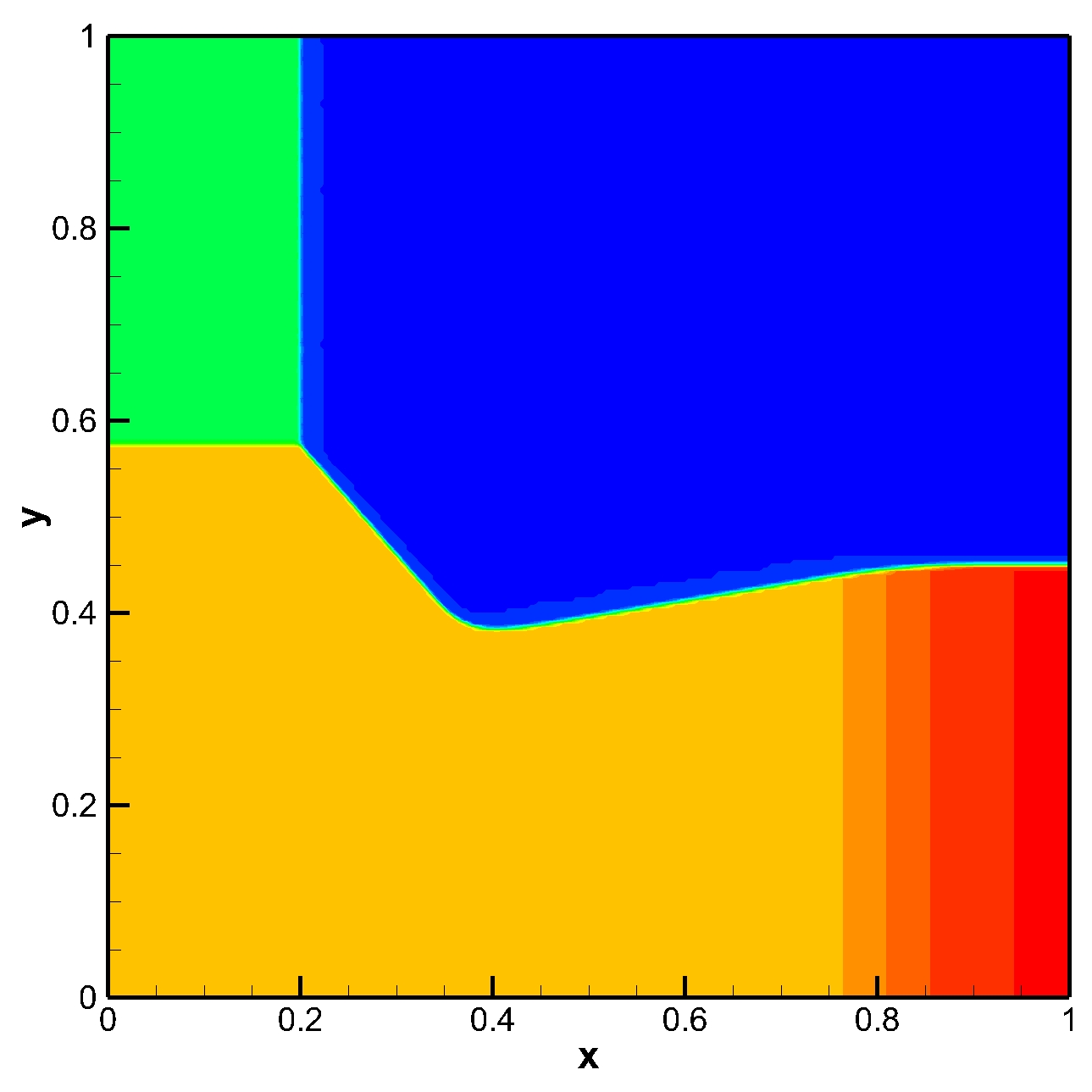}
		\caption{ $\mathbb{P}^3$: {\tt optimal} approach.}
	\end{subfigure}
	\hfill
	\begin{subfigure}{0.31\textwidth}
		\includegraphics[width=\textwidth]{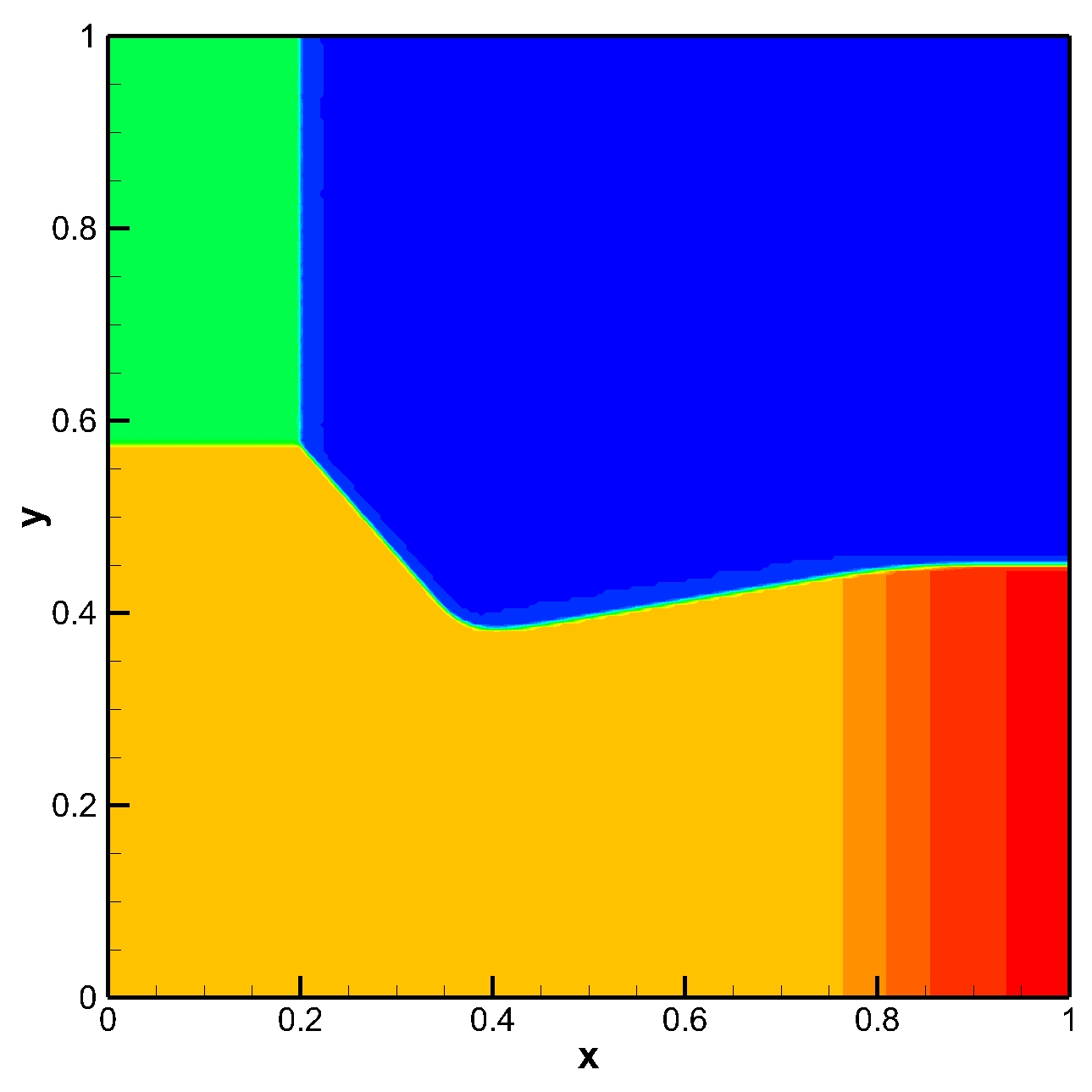}
		\caption{$\mathbb{P}^3$: {\tt classic} approach.}
	\end{subfigure}
	\hfill
	\begin{subfigure}{0.35\textwidth}
		\includegraphics[width=\textwidth]{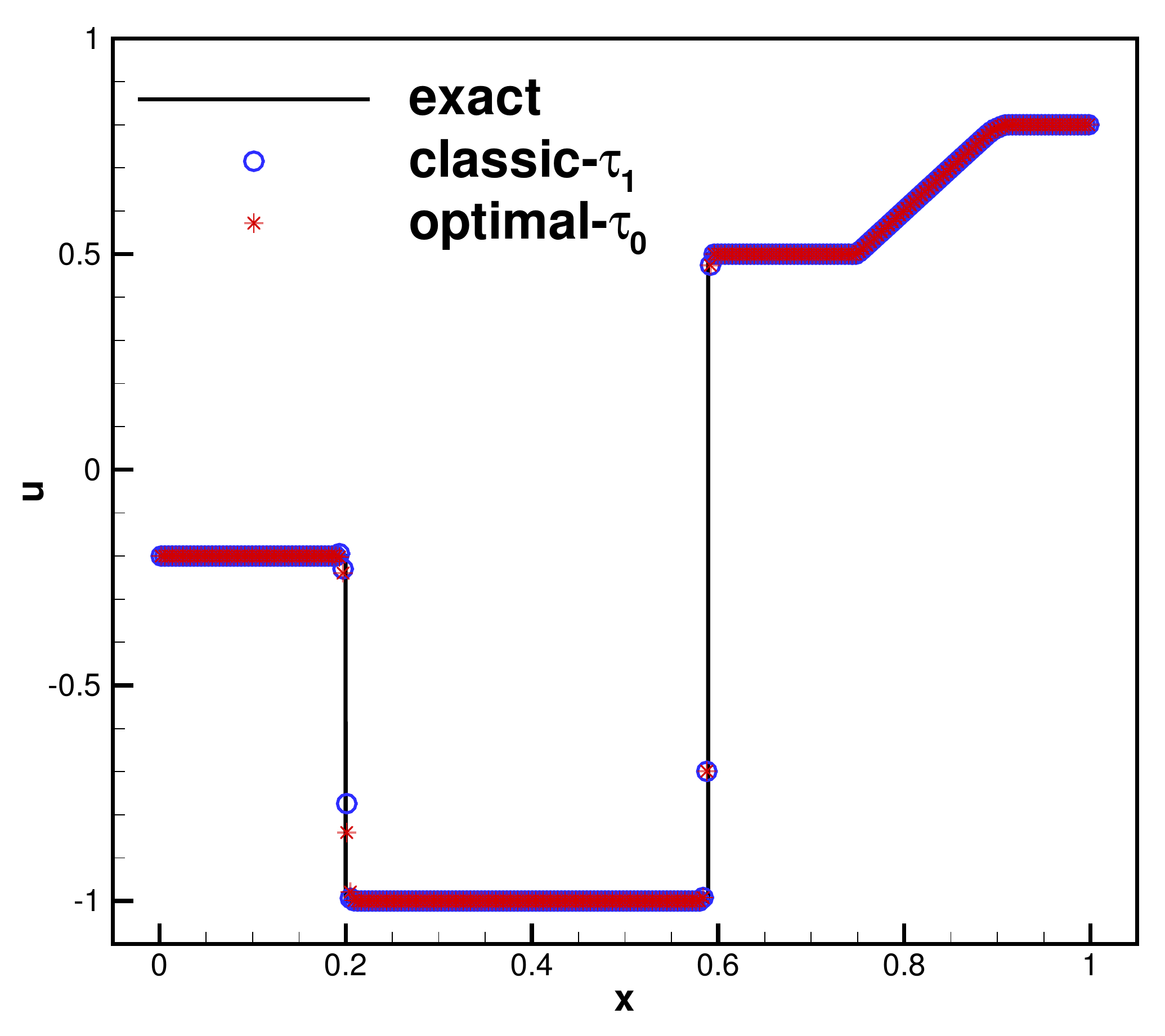}
		\caption{$\mathbb{P}^3$: cut along the diagonal line  $x=1-y$.}
	\end{subfigure}

	\caption{Example \ref{Ex3}: The numerical solutions at $t=0.5$ obtained via the {\tt optimal} approach and the {\tt classic} approach.  
	}
	\label{fig:ex3_1}
\end{figure} 

\end{example}

\begin{table}[htb] 
	\centering
	\caption{CPU time in seconds for Examples \ref{Ex3}--\ref{Ex5}.}
	\label{tab:ex-cpu}
	\setlength{\tabcolsep}{5.6mm}{
		\begin{tabular}{lccccc}
			\toprule[1.5pt]
			
			& \multirow{2}{*}{Approach} 
			
			& Example \ref{Ex3}  & Example \ref{Ex4}      & Example \ref{Ex5} \\
			& & 2D Riemann problem  &  Mach 80 jet & Mach 2000 jet \\
			
			\midrule[1.5pt]
			\multirow{2}{*}{ $\mathbb{P}^2$} 
			& optimal-$\tau_0$   & 309.39  & 13297.97 & 10708.38 \\
			& classic-$\tau_1$   & 504.95 & 20501.56 & 16527.88 \\
			
			\midrule[1.5pt]
			\multirow{2}{*}{ $\mathbb{P}^3$} 
			& optimal-$\tau_0$   & 664.71 & 30861.78 & 24304.09 \\
			& classic-$\tau_1$   & 1106.00& 42714.91 & 34896.86 \\
			
			\bottomrule[1.5pt]
		\end{tabular}
	}
\end{table}

\begin{example}[Mach 80 jet problem of Euler equations]\label{Ex4}
	This example simulates a Mach 80 jet \cite{ha2005numerical,zhang2010b,liu2022essentially} by solving the two-dimensional Euler equations, 
	which can be 
	written in the form of (\ref{2DCL}) with
	\begin{equation} \label{Eq:Euler}
			u=\begin{pmatrix}
				\rho \\
				\rho v_1 \\
				\rho v_2 \\
				E
			\end{pmatrix}, \qquad {f_1}(u)=\begin{pmatrix}
				\rho v_1 \\
				\rho {v_1}^{2}+p \\
				\rho v_1 v_2 \\
				(E+p) v_1
			\end{pmatrix}, \qquad {f_2}(u)=\begin{pmatrix}
				\rho v_2 \\
				\rho v_1 v_2 \\
				\rho {v_2}^{2}+p \\
				(E+p) v_2
		\end{pmatrix}
	\end{equation}
with $E=\frac{1}{2} \rho ( {v_1}^{2} + {v_2}^{2})+\rho e$ and $p=(\gamma-1) \rho e$.  
Here, $\rho$ is the density, $(v_1,v_2)$ denotes the velocity, $p$ is the pressure, $E$ is the total energy, and $e$ is the specific internal energy. The ratio of specific heats $\gamma$ is set to be $5/3$. 
The density and pressure should be positive, yielding the invariant region 
$
G=\{ u: \rho>0, p>0 \}, 
$ which is a convex set \cite{zhang2010b}.

Initially, 
the computation domain $[0,2]\times[-0.5,0.5]$ is  full of the ambient gas with $(\rho,v_1,v_2,p) = (5,0,0,0.4127)$. 
The jet with state $(\rho,v_1,v_2,p) = (5,30,0,0.4127)$ is injected into the domain from the left boundary between $y=-0.05$ and $0.05$. All the other boundaries are set as  outflow boundary conditions, as in \cite{ha2005numerical,liu2022essentially}. 
This is a benchmark yet challenging test, and a high-order numerical scheme without any BP techniques may easily produce negative density and/or negative pressure, which  
eventually causes the breakdown of the simulation code. 
We perform the simulation until $t=0.07$ on the uniform mesh of $480\times 240$ cells. 
The numerical results of density are presented in Figure \ref{fig:ex4_1}, from which we 
	clearly observe that the critical features of the jet: cocoons, bow shock, shear flows, etc. 
	All these features 
	are well captured by the BP DG methods and agree with those presented in \cite{ha2005numerical,liu2022essentially}. 
	The results of the {\tt optimal} 
	approach with time step $\tau_0=C_{\tt SSP}  \tau^{\tt BP}_{\tt O}$ 
	are comparable  to those of 
	 the {\tt classic} approach with time step $\tau_1=C_{\tt SSP}  \tau^{\tt BP}_{\tt C}$. 
	As shown in Table \ref{tab:ex-cpu}, the {\tt optimal} approach  allows a larger time step and 
	takes much less CPU time than the {\tt classic} approach. 
	
For both approaches, 
the local scaling BP limiter \cite{zhang2010b} is necessary to enforce the conditions (\ref{eq:478a}) and (\ref{eq:478b}). Due to the presence of strong shocks in this and next examples, the WENO limiter \cite{qiu2005runge} is also used, right before the BP limiter, within some adaptively detected troubled cells to suppress potential numerical oscillations.



\begin{figure}[h]
	\centering
	\begin{subfigure}[b]{0.49\textwidth}
		\begin{center}
			\includegraphics[width=\textwidth]{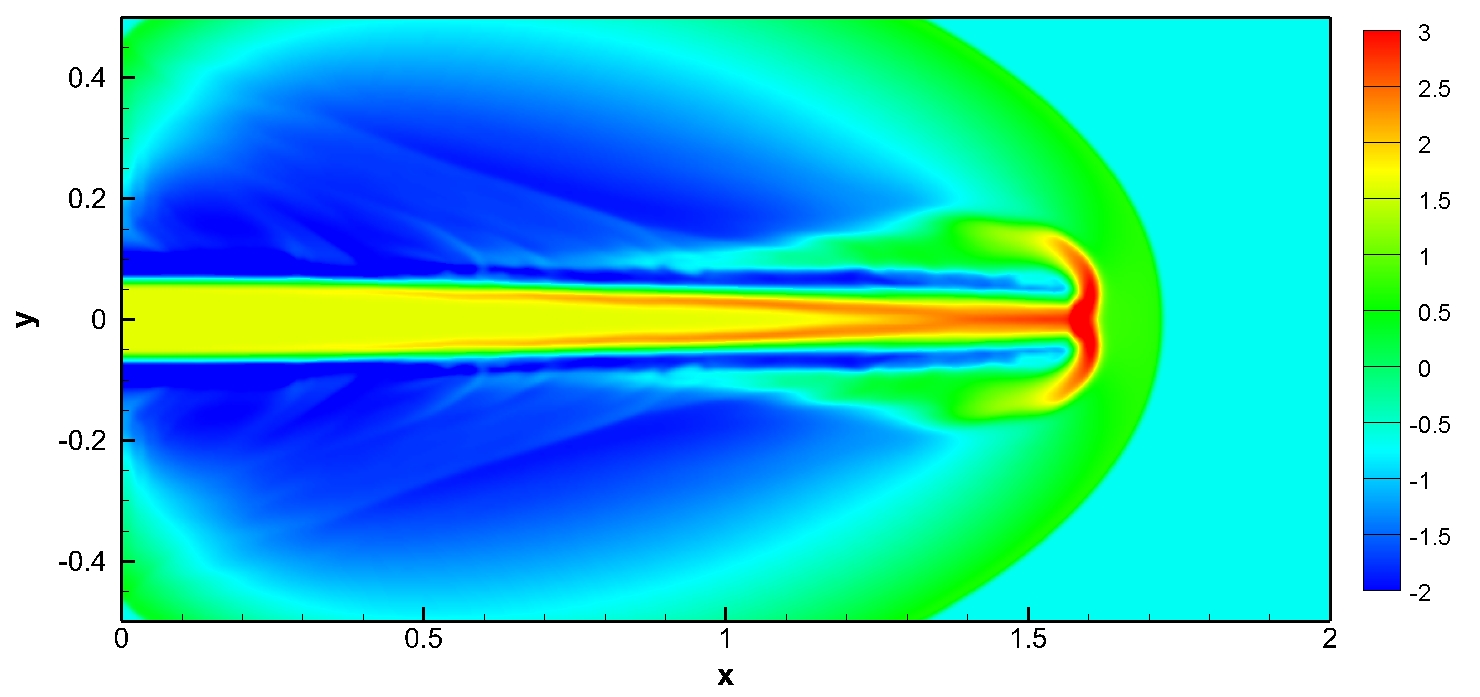}
			\caption{$\mathbb{P}^2$: {\tt optimal} approach.}
		\end{center}
	\end{subfigure}
	\hfill
	\begin{subfigure}[b]{0.49\textwidth}
		\begin{center}
			\includegraphics[width=\textwidth]{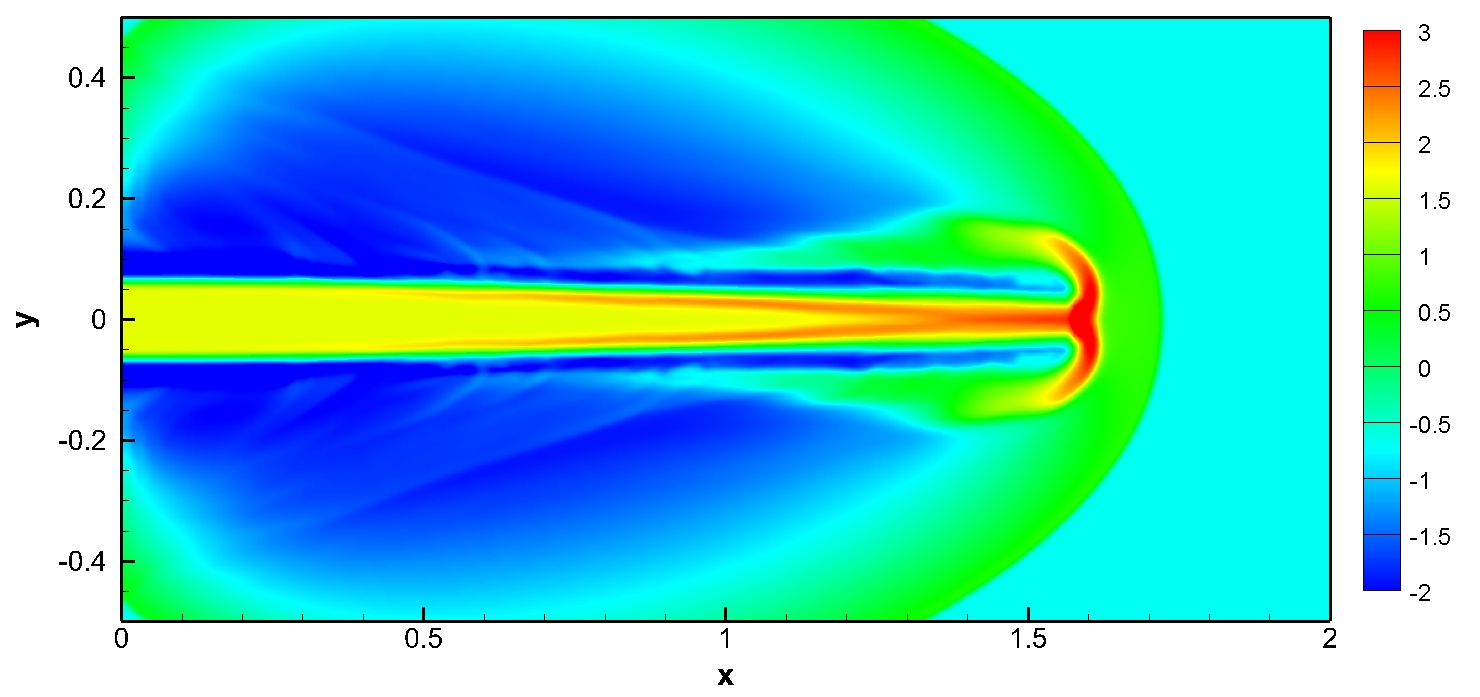}
			\caption{$\mathbb{P}^2$: {\tt classic} approach.}
		\end{center}
	\end{subfigure}
	
	\begin{subfigure}[b]{0.49\textwidth}
		\begin{center}
			\includegraphics[width=\textwidth]{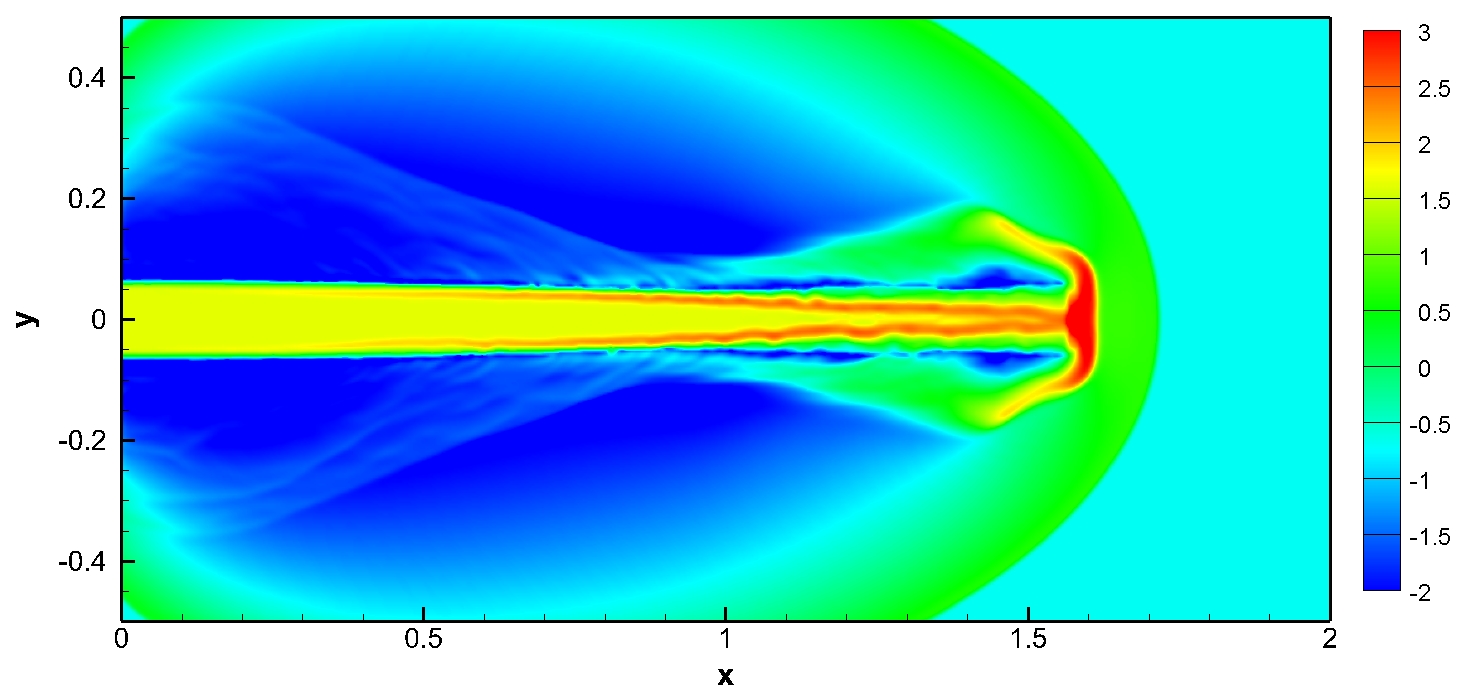}
			\caption{$\mathbb{P}^3$: {\tt optimal} approach.}
		\end{center}
	\end{subfigure}
	\hfill
	\begin{subfigure}[b]{0.49\textwidth}
		\begin{center}
			\includegraphics[width=\textwidth]{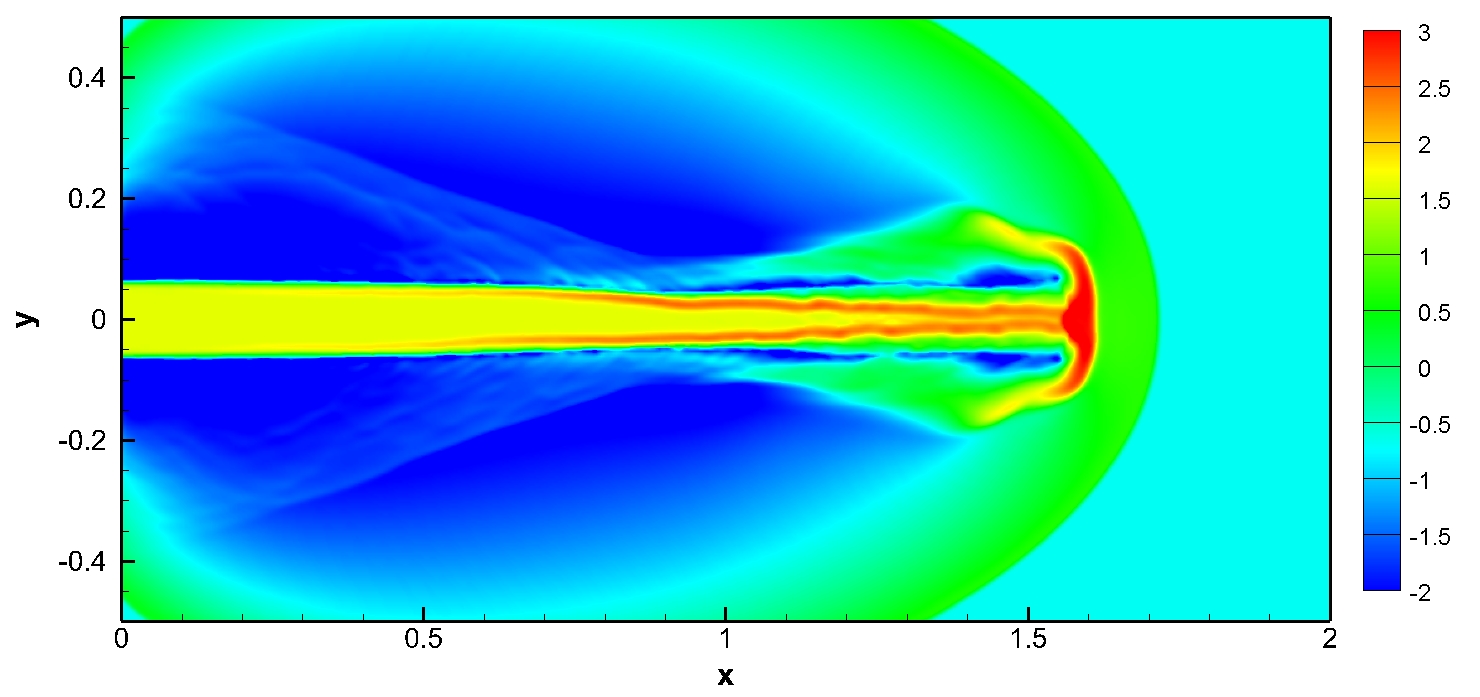}
			\caption{$\mathbb{P}^3$: {\tt classic} approach.}
		\end{center}
	\end{subfigure}
	
	\caption{Example \ref{Ex4}: The numerical density at $t=0.07$ obtained by using the third-order (top row) 
		and the fourth-order (bottom row) BP DG schemes designed via the {\tt optimal} approach and the {\tt classic} approach.}
	\label{fig:ex4_1}
\end{figure}

\begin{figure}[h]
	\centering
	\begin{subfigure}[b]{0.49\textwidth}
		\begin{center}
			\includegraphics[width=\textwidth]{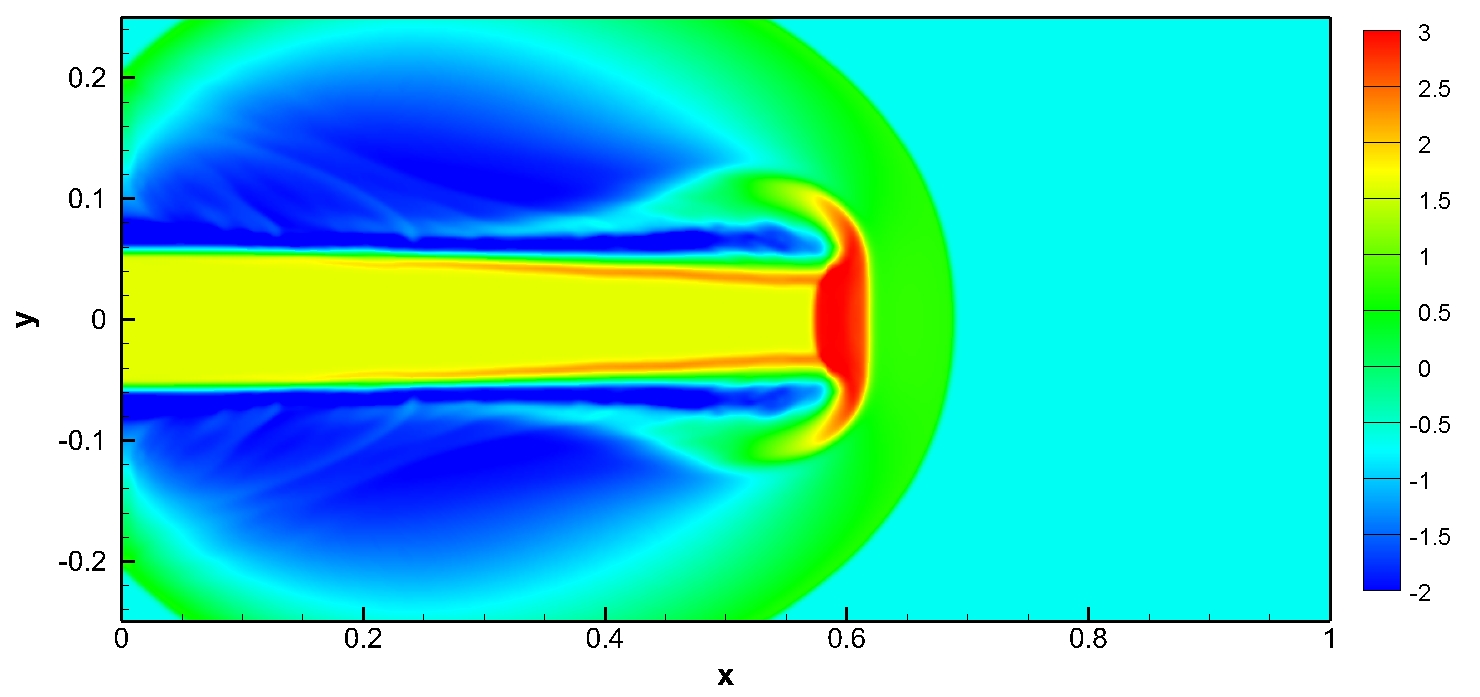}
			\caption{$\mathbb{P}^2$: {\tt optimal} approach.}
		\end{center}
	\end{subfigure}
	\hfill
	\begin{subfigure}[b]{0.49\textwidth}
		\begin{center}
			\includegraphics[width=\textwidth]{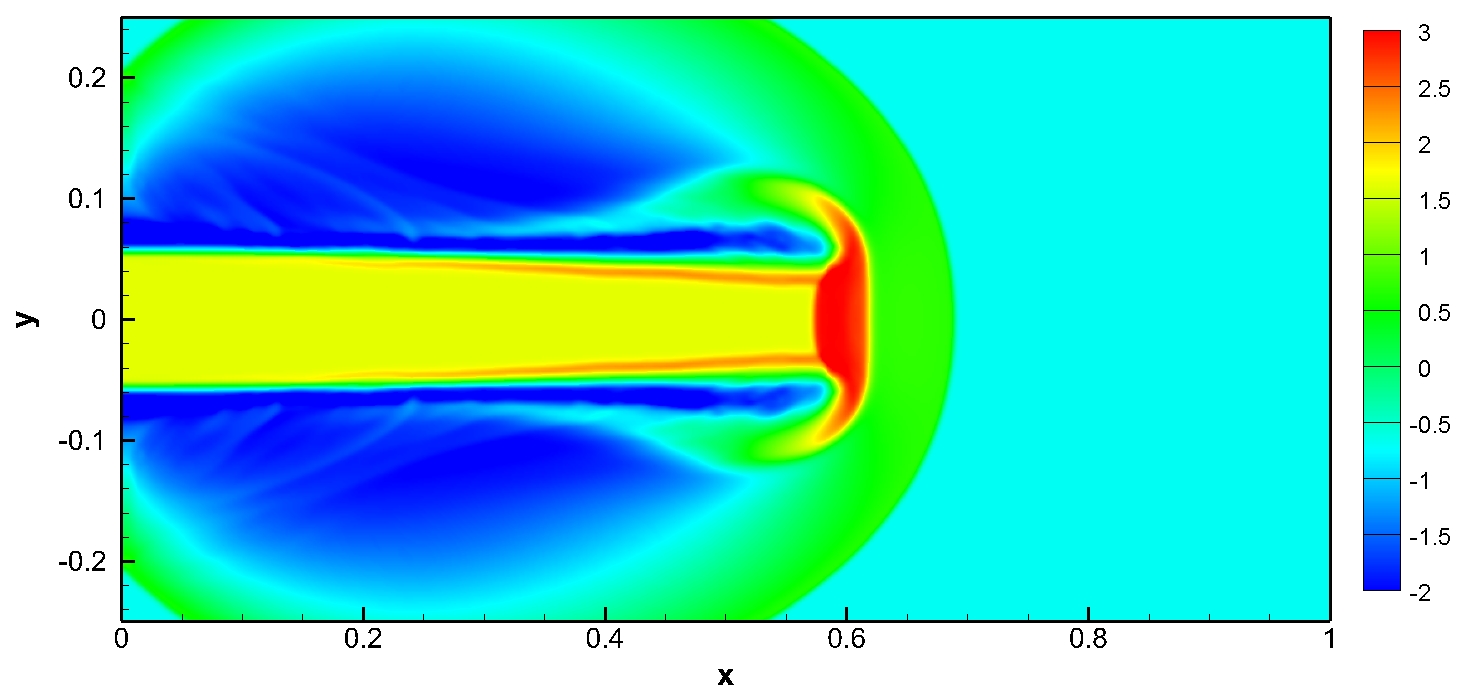}
			\caption{$\mathbb{P}^2$: {\tt classic} approach.}
		\end{center}
	\end{subfigure}
	
	\begin{subfigure}[b]{0.49\textwidth}
		\begin{center}
			\includegraphics[width=\textwidth]{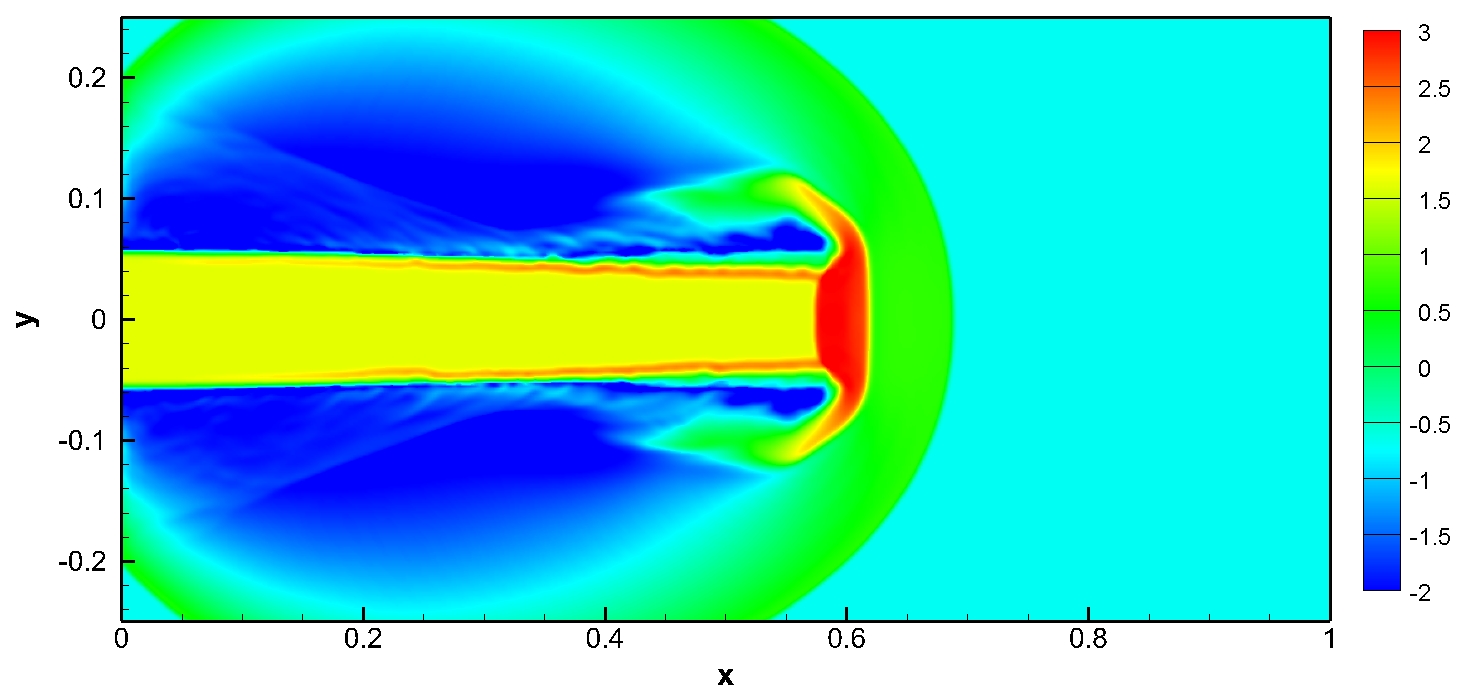}
			\caption{$\mathbb{P}^3$: {\tt optimal} approach.}
		\end{center}
	\end{subfigure}
	\hfill
	\begin{subfigure}[b]{0.49\textwidth}
		\begin{center}
			\includegraphics[width=\textwidth]{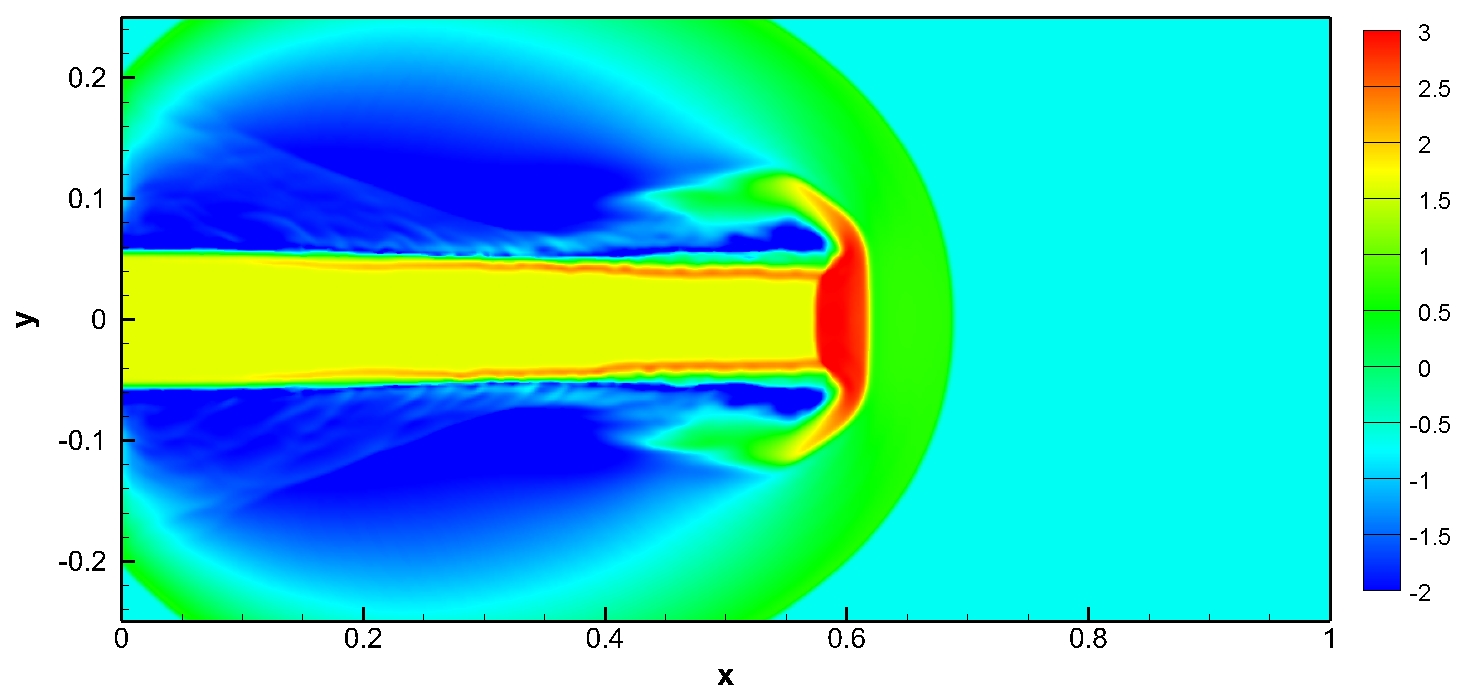}
			\caption{$\mathbb{P}^3$: {\tt classic} approach.}
		\end{center}
	\end{subfigure}
	
	\caption{Same as Figure \ref{fig:ex4_1} except for Example \ref{Ex5} (Mach 2000 jet) at $t=0.001$.}
	\label{fig:ex4_2}
\end{figure}

\end{example}

\begin{example}[Mach 2000 jet problem of Euler equations]\label{Ex5}
Last, a Mach 2000 jet is considered with the Euler equation (\ref{Eq:Euler}) in the domain $[0,1]\times[-0.25,0.25]$.
The setup is the same as Example \ref{Ex4}, except the jet state fixed as $(\rho,v_1,v_2,p) = (5,800,0,0.4127)$ for  $y\in[-0.05,0.05]$ on the left boundary ($x=0$). 
All the other boundaries are of outflow conditions. 
The much higher Mach number renders this jet test more challenging than Example \ref{Ex4}. 
The simulation is performed until $t=0.001$ with $480\times 240$ cells.
Figure \ref{fig:ex4_2} presents the numerical results of density, which demonstrate the comparably high resolution and excellent robustness for both the {\tt optimal} approach and the {\tt classic} approach. 
Table \ref{tab:ex-cpu} also displays the CPU time in this test for both approaches, further confirming the notable advantage of the {\tt optimal} approach in efficiency. 

\end{example}

\section{Summary}

In this paper, we proposed the problem of seeking the optimal  
convex decomposition of the cell average for 
 constructing high-order BP schemes of hyperbolic conservation laws in multiple dimensions within the Zhang--Shu framework. 
It was observed that the classic Zhang--Shu convex decomposition, based on the tensor product of Gauss--Lobatto and Gauss quadratures, is generally not optimal in the multidimensional cases. 
For the $\mathbb P^2$ and $\mathbb P^3$ spaces, which are typically used in the third-order and fourth-order DG schemes, we discovered the optimal convex decomposition that achieves the mildest BP CFL condition yet requires much fewer internal nodes.  
Based on our optimal convex decomposition, we established more efficient high-order BP schemes, which allow a larger BP time step-size than the classic one, in the Zhang--Shu framework.  
We presented several numerical examples to validate the remarkable advantages of using our optimal decomposition over the classic one in terms of efficiency.

The discovery of the optimal convex decomposition was quite nontrivial and might have a broad impact, as it would 
lead to an overall improvement of third-order and fourth-order BP schemes for  
a large class of hyperbolic or convection-dominated equations at the cost of only a slight and local modification to the implementation code. 
Our work in this paper was limited to the multivariate polynomial spaces $\mathbb{P}^2$ and $\mathbb{P}^3$. In more general cases, 
many questions about the optimal convex decomposition are yet open; for example,   
what are the optimal decompositions for more general polynomial spaces $\mathbb{P}^k$ with $k\ge 4$ on Cartesian meshes,  triangular meshes, and more general unstructured meshes? 
We 
hope this paper could motivate further exploration along this direction in the future.


\bibliographystyle{unsrt} 
\bibliography{refs}

\end{document}